\documentclass[11pt]{amsart}
\usepackage{amssymb, latexsym, amsmath, amsfonts, tikz, amscd}
\usepackage{tikz-cd} 
\usepackage{hyperref, enumitem, fancyhdr}
\usepackage{color}
\newtheorem{thm}{Theorem}[section]
\newtheorem{cor}[thm]{Corollary}
\newtheorem{lem}[thm]{Lemma}
\newtheorem{prop}[thm]{Proposition}
\theoremstyle{definition}

\theoremstyle{remark}
\newtheorem{rem}[thm]{Remark}
\numberwithin{equation}{section}
\theoremstyle{remark}
\newtheorem{exam}[thm]{Example}

\newcommand{\mbb}{\mathbb}
\newcommand{\ra}{\rightarrow}
\newcommand{\z}{\zeta}
\newcommand{\pa}{\partial}
\newcommand{\ov}{\overline}
\newcommand{\sm}{\setminus}

\newcommand{\no}{\noindent}
\newcommand{\al}{\alpha}
\newcommand{\Om}{\Omega}
\newcommand{\cal}{\mathcal}
\newcommand{\ti}{\tilde}

\newcommand{\de}{\delta}

\newcommand{\om}{\omega}
\newcommand{\ga}{\gamma}

\newcommand{\be}{\beta}

\begin{document}
\title{On the automorphism group of certain Short $\mbb C^2$'s}
\author{Sayani Bera, Ratna Pal and Kaushal Verma}

\address{SB: Indian Association for the Cultivation of Science, 2A-2B Raja S. C. Mullick Road, Kolkata 700 032, India}
\email{sayanibera2016@gmail.com, mcssb2@iacs.res.in}

\address{RP:  Indian Institute of Science Education and Research Mohali, Knowledge City, Sector -81, Mohali, Punjab.-140306, India}
\email{ratna.math@gmail.com, ratnapal@iisermohali.ac.in}

\address{KV: Department of Mathematics, Indian Institute of Science, Bangalore 560 012, India}
\email{kverma@iisc.ac.in}

\begin{abstract} 
 For a H\'enon map of the form $H(x, y) = (y, p(y) - ax)$, where $p$ is a polynomial of degree at least two and $a \not= 0$, it is known that the sub-level sets of the Green's function $G^+_H$ associated with $H$ are Short $\mbb C^2$'s. For a given $c > 0$, we study the holomorphic automorphism group of such a Short $\mbb C^2$, namely $\Om_c = \{ G^+_H < c \}$. The unbounded domain $\Om_c \subset \mbb C^2$ is known to have smooth real analytic Levi-flat boundary. Despite the fact that $\Om_c$ admits an exhaustion by biholomorphic images of the unit ball, it turns out that its automorphism group, ${\rm Aut}(\Om_c)$ cannot be too large. On the other hand, examples are provided to show that these automorphism groups are non-trivial in general. We also obtain necessary and sufficient conditions for such a pair of Short $\mbb C^2$'s to be biholomorphic. 
\end{abstract}  

\maketitle 

\section{Introduction}

\no Examples of domains $\Om \subset \mbb C^2$ that can be written as an increasing sequence of biholomorphic images of the unit ball $\mbb B^2 \subset \mbb C^2$, and at the same time admit a non-constant bounded plurisubharmonic function and on which the infinitesimal Kobayashi metric vanishes identically were first constructed by Fornaess in \cite{F} and christened {\it Short}\;$\mbb C^2$'s.

\medskip

There are two principal methods for constructing such domains which can be summarized as:

\medskip

\noindent (i) Fix an integer $d \ge 2$ and a sequence $1 > a_1 \ge a_2 \ge \ldots \ge \lim a_n = a_{\infty} \ge 0$. Let $\{F_n\}$, $n \ge 1$, be a sequence of automorphisms of $\mbb C^2$ of the form
    \[
    F_n(x, y) = (\eta_n y + x^d, \eta_n x)
    \]
    where $\eta_n = a^{d^n}_n$. For $i = 1, 2$, let $\pi_i : \mbb C^2 \ra \mbb C$ be the standard projection on the $i$-th coordinate. Define
    \[
    \phi_n(x,y) = \max \{ \vert \pi_i \circ F_n \circ F_{n-1} \circ \cdots \circ F_1(x,y) \vert, \eta_n  : i = 1, 2 \} 
    \]
    and set
    \[
    \phi(x,y) = \lim_{n \ra \infty} \frac{1}{d^n} \log \phi_n(x,y)
    \]
    which can be shown to be plurisubharmonic on $\mbb C^2$. Then, for every $c > \log a_{\infty}$, the $c$ sub-level set of $\phi$ given by $\{\phi < c\}$ is a {\it Short}\;$\mbb C^2$. Note that the $0$ sub-level set can be identified precisely with the basin of attraction of the given sequence $\{F_n\}$.

\medskip

\noindent (ii) Recall that for a H\'{e}non map of the form $H(x,y) = (y, p(y) - ax)$, where
$a \not= 0$ and $p$ is a polynomial of degree $d\geq 2$, the Green's function
    \[
    G^+_H(x,y) = \lim_{n \ra \infty} \frac{1}{d^n} \log^+ \Vert H^n(x,y) \Vert
    \]
    is a continuous, non-negative plurisubharmonic function on $\mbb C^2$ that describes the asymptotic growth of orbits. Then, for every $c > 0$, the $c$ sub-level set 
    $\{G_H^+ < c\}$ is a {\it Short}\;$\mbb C^2$.

\medskip

Much like Fatou--Bieberbach domains, {\it Short}\;$\mbb C^2$'s come in a variety of shapes and sizes, and possess a number of intriguing properties as can be seen from the examples in \cite{ATP}, \cite{B}, \cite{BPV1}, \cite{BPV:Rigidity}, \cite{T} and \cite{ForP}. The purpose of this paper is to study the holomorphic automorphism group of {\it Short}\;$\mbb C^2$'s that arise in (ii) above. Apart from its intrinsic interest, part of the motivation for doing so comes from a more general question namely, understanding the dependence of the automorphism group ${\rm Aut}(D)$ on the domain $D \subset \mbb C^n$. Suitably interpreted, the assignment $D \mapsto {\rm Aut}(D)$ is upper semi-continuous within the category of {\it bounded} domains. The circle of ideas contained in this meta-theorem of sorts are due to Greene--Kim--Krantz (see \cite{GK, GKK} for example, and \cite{GKKS}, \cite{Kr} for recent work in this direction) and Fridman--Ma--Poletsky (see \cite{FM, FMP, FP}). For {\it unbounded} domains, one can begin by considering a simple example. Take an exhaustion of $\mbb C^n$, $n \ge 2$, by an increasing union of concentric balls $B(0, r_j) \subset \mbb C^n$ with $r_j \ra \infty$. Then ${\rm Aut}(B(0, r_j)) \simeq SU(n,1)$ for all $j$, but ${\rm Aut}(\mbb C^n)$ is infinite dimensional. Thus, there is a large upward jump in the automorphism group of the limiting domain as compared to the domains approximating it. The same phenomenon is seen in the case of a Fatou--Bieberbach domain $D \subset \mbb C^n$. 
Indeed, such a domain $D$ is biholomorphic to $\mbb C^n$ and admits an exhaustion by domains that are biholomorphic to the unit ball $\mbb B^n \subset \mbb C^n$. 

\medskip

By their very construction, {\it Short}\;$\mbb C^2$'s also admit an exhaustion by domains that are biholomorphic to $\mbb B^2$. It is therefore of interest to understand their automorphism group -- an aspect of theirs that seems not to have been explored so far. It turns out that the automorphism group of a {\it Short}\;$\mbb C^2$ that arises by considering a single H\'{e}non map as in (ii) is surprisingly small. This will provide an example of the limitations of the upper-semicontinuity phenomenon for unbounded domains.  However,  it must be emphasized that the automorphism group need not be trivial.

\begin{exam}
Let $H(x,y)=(y, y^2-ax)$ where $a \neq 0$ and $L_\omega=(\omega x, \omega^2 y)$ where $\omega$ is  the primitive cube root of identity, i.e., $\omega^3=1$. Note that 
\[
L_{\omega} \circ H(x,y)=(\om y, \omega^2 y^2-\omega^2 a x)=H \circ L_{\omega^2} (x,y)
\]  
for every $(x,y) \in \mathbb{C}^2$. Since $L_{\omega} \circ L_{\omega^2} = L_{\omega^2} \circ L_{\omega} =$ identity, it follows that $H = L_{\omega}\circ H \circ L_{\omega}$ and hence $H^n = L_{\omega^2} \circ H^n \circ L_{\omega}$ or $L_{\omega} \circ H^n \circ L_{\omega}$ if $n$ is even or odd respectively. Thus, $G^+_H$, the Green's function of $H$, satisfies $G_H^+=G_H^+\circ L_\omega$. Further, for every $c>0, L_\omega$ is an automorphism of the $c$-sublevel set of $G_H^+$, i.e., $\{G_H^+<c\}$.
\end{exam}

To make all this precise, let us begin by recalling some relevant facts about the dynamics of a H\'{e}non map from \cite{BS1}, \cite{HOV} and fixing notation.

\medskip

Consider a H\'{e}non map of the form $H(x, y) = (y, p(y) - ax)$ where $a \not= 0$ and $p$ is a monic polynomial of degree $d \ge 2$. By an affine conjugacy, the coefficient of $y^{d-1}$ in $p$ can be made zero, and as a result, we will henceforth assume that 
\begin{equation}\label{1.1}
   H(x, y) = (y, p(y) - ax) 
\end{equation}
where $p(y)  = y^d + q(y)$ with $q(y) = a_{d-2}y^{d-2} + \cdots + a_0$. It is useful to understand H\'{e}non maps in terms of their behaviour at infinity, and to do this, consider the filtration of $\mbb C^2$ given by 
\[
V_R = \left\{ \vert x \vert, \vert y \vert \le R \right\}, V^{+}_R = \left\{\vert y \vert \ge \max\{\vert x \vert, R\} \right\},\; V^-_R = \left\{ \vert x \vert \ge \max\{\vert y \vert, R\} \right\}. 
\]
There exists a large $R > 1$ for which $H(V^+_R) \subset V^+_R$ and the forward orbit of each point in $V^+_R$ escapes to infinity, $H(V_R \cup V^+_R) \subset V_R \cup V^+_R$, and the forward orbit of points in $V^-_R $, with unbounded (forward) orbit, is eventually contained $V_R \cup V^+_R$. Bounded forward orbits are therefore trapped in $V_R \cup V_R^-$, and the unbounded ones approach infinity while remaining in $V^+_R$. Let
\[
K^{\pm} = \{(x, y) \in \mbb C^2: \{H^{\pm n}(x,y), n \ge 0\} \;\text{is bounded} \}
\]
and
\[
K = K^{+} \cap K^{-}, \;\;J^{\pm} = \pa K^{\pm}, \;\;J = J^+ \cap J^-, \;\;U^\pm = \mbb C^2 \sm K^{\pm}.
\]
Let $ (x_n, y_n)=H^n(x,y)$ where $H^n$ denote the $n$-th iteration of $H$. Since $\deg q \le d-2$, (\ref{1.1}) shows that the degree of $x_n$ is strictly dominated by the degree of $y_n$ for all $n \ge 1$ and that $y_{n+1} \sim y^d_n$. Hence, the limit (which defines the Green's function)
\[
G_H^+(x, y) = \lim_{n \ra \infty} \frac{1}{d^n} \log^+ \Vert H^n(x,y) \Vert = \lim_{n \ra \infty} \frac{1}{d^n} \log^+ \vert y_n \vert
\]
converges uniformly on compact subsets of $\mbb C^2$ and defines a continuous, non-negative plurisubharmonic function on $\mbb C^2$. Further, $G_H^+$ is pluriharmonic on $U^+$, $\{G^+_H = 0\} = K^+$, and the definition shows that $G_H^+ \circ H = d \cdot G_H^+$. Since $y_{n+1} = y^d_{n}\left(1 + ({q(x_n, y_n)}/{y^d_n}) \right) \sim y^d_n$ for $(x, y) \in V^+_R$, the B\"{o}ttcher coordinate 
\[
\phi^+(x, y) = y \lim_{n \ra \infty} \prod_{j=0}^{n-1} \left( 1 + \frac{q(x_j, y_j)}{y^d_j}  \right)^{1/d^{j+1}}
\]
is a well defined holomorphic function on $V^+_R$ (the limit is uniform on compact subsets of $V^+_R$) that satisfies $\phi^+ \circ H = (\phi^+)^d$. In addition, $G_H^+ = \log \vert \phi^+ \vert$ on $V^+_R$. Furthermore, $G_H^+ : U^+ \ra \mbb R_+$ is a pluriharmonic submersion whose level sets are smooth $3$-manifolds that admit a foliation by copies of 
$\mbb C$. Each such leaf, which is a copy of $\mbb C$, is dense in the level set which contains it. In particular, all this applies to the boundary of the {\it Short}\;$\mbb C^2$
\begin{equation}
 \Om_c = \{G_H^+ < c\}.   
\end{equation}
It is useful to recall that $\phi^+$ extends as a multi-valued map to $U^+$. To see this, let $\ga \subset U^+$ be a path which starts at a point of $V^+_R$. Let $G^+_{H\ast}$ be a pluriharmonic conjugate for $G_H^+$ in a neighbourhood of the initial point of $\ga$ such that $\phi^+ = \exp(G_H^+ + i G^+_{H\ast})$. By continuing $G^+_{H\ast}$ along $\ga$ in a pluriharmonic manner, it follows that $\phi^+$ admits analytic continuation along $\ga$, and hence to all of $U^+$. Thus, $\phi^+ : U^+ \ra \mbb C \sm \overline{\mbb D}$ is a multi-valued holomorphic map; here $\mbb D \subset \mbb C$ is the unit disc. Let $\ti U^+$ be the Riemann domain over $U^+$ on which $\phi^+$ lifts as a single-valued holomorphic map $\ti \phi^+ : \ti U^+ \ra \mbb C \sm \overline{\mbb D}$. The fibres of $\ti \phi^+$ are precisely the leaves of the foliation of the level sets of $G_H^+$, and are hence equivalent to $\mbb C$. Furthermore, there is a biholomorphic map
\[
\tau : \mbb C \times (\mbb C \sm \overline{\mbb D})  \ra \ti U^+
\]
such that $\ti \phi^+ \circ \tau( t,\zeta) = \zeta$. In other words, $\tau$ straightens out the fibres of $\ti \phi^+$ globally, and in these coordinates, the fibres are described by $\zeta = \text{constant}$. The map $H : U^+ \ra U^+$ lifts to $\ti H :  \mbb C \times (\mbb C \sm \overline{\mbb D})   \ra  \mbb C \times (\mbb C \sm \overline{\mbb D}) $. For the H\'{e}non map $H(x, y) = (y, y^2 + a_0 - ax)$, an explicit description of the lift as
\begin{equation}
\ti H(z,\zeta) = \left( (a/2)z + \z^3 - (a_0/2) \z, \z^2 \right)  
\end{equation}
can be found in \cite{HOV}. Taking this further, Bousch \cite{Bo} showed that every automorphism of $U^+$ that induces the identity on $\pi_1(U^+) \simeq \mathbb Z\left[1/2\right]$ lifts to an automorphism of $\mbb C \times (\mbb C \sm \overline{\mbb D})$, and since these can be written down explicitly, a detailed description of \text{Aut}$(U^+)$ follows. Indeed, \text{Aut}$(U^+)$ is isomorphic to the semi-direct products $\mbb C \rtimes \mbb Z$ or $\mbb C \rtimes (\mbb Z/3 \mbb Z) \rtimes \mbb Z$ according as $a_0 \not= 0$ or $a_0 = 0$ respectively. In this context, consider the {\it punctured} {\it Short}\;$\mbb C^2$
\[
\Om'_c = \Om_c \sm K^+ = \{0 < G_H^+ < c\} \subset \Om_c
\]
which can be shown to be connected for every $c > 0$, and for every $f \in {\rm Aut}(\Om_c)$, its restriction $f : \Om'_c \ra \Om_c$. 

\begin{prop}\label{Prop 1}
	Let $c>0$ and $f \in  {\rm Aut}(\Omega_{c})$. Then $f(K^+)=K^+$ and $f(\Omega_b)=\Omega_b$ for every $0<b<c$. In particular, the restriction of $f$ to $\Om'_c$ is in fact an element of ${\rm Aut}(\Om'_c)$
	\end{prop}

This shows that the restriction map $r : {\rm Aut}(\Om_c) \ra {\rm Aut}(\Om'_c)$ defined by $r(f) = f|_{\Om'_c}$ for $f \in {\rm Aut}(\Om_c)$ is well defined. It is clearly a homomorphism and injective as well by the uniqueness theorem. Therefore, to understand how large ${\rm Aut}(\Om_c)$ can be, it suffices to get a hold on ${\rm Aut}(\Om'_c)$.

\begin{thm}\label{AutOmegac}
Let $H(x, y) = (y, p(y) - ax)$ be a H\'{e}non map of degree $d$ as in (\ref{1.1}). For $c > 0$, consider the punctured {\it Short}\;$\mbb C^2$ defined by $\Om'_c = \{0 < G_H^+ < c\}$. Then the fundamental group of $\Om_c'$, $\pi_1(\Om'_c)$ is $\mathbb Z\left[1/d\right]$. Furthermore, $\mathbb{C}\subseteq \rm {Aut}(\Om'_c)   \subseteq \left(\mathbb{Z}_{d - 1}\ltimes \mathbb{C}\right)\times \mathbb{Z}_{d+1}$.   Consequently, ${\rm Aut}(\Om_c) \subseteq \left(\mathbb{Z}_{d - 1}\ltimes \mathbb{C}\right)\times \mathbb{Z}_{d+1}$.   
\end{thm}

The computation of $\pi_1(\Om'_c)$ is along similar lines as $\pi_1(U^+)$, but there is one essential difference -- $U^+$ is invariant under $H$, but $\Om'_c$ is not. However, the redeeming feature is that $H : \Om'_c \ra \Om'_{cd}$ is biholomorphism (since $G_H^+ \circ H = d \cdot G_H^+$) and a systematic use of this allows the computations to go through. For the automorphism group of $\Om'_c$, the idea is to construct a cover of $\Om'_c$ to which $\phi^+$ extends as a single-valued holomorphic map, and to identify it up to biholomorphism as the product of  $\mbb C$ and an annulus, namely $ \mbb C \times \cal A_c $, where $\cal A_c = \{1 < \vert \z \vert < e^c\}$. It turns out that any automorphism $a$ of $\Omega_c'$ induces an automorphism $A$ of $ \mbb C \times \cal A_c $.  A careful analysis of the calculations in \cite{HOV} leading up to (1.3) can be found in Bonnot--Radu--Tanase \cite{BRT} which allows them to identify the lift $\ti H$ (as an automorphism of $\mbb C  \times (\mbb C \sm \overline{\mbb D})$) corresponding to $H$ in (1.1) as follows:
\begin{equation}\label{H tilde int}
\ti H( z,\zeta) = \left((a/d)z + Q(\zeta), \zeta^d\right) 
\end{equation}
 where
\[
Q(\z) = \z^{d+1} - (a_{d-2}/d) \z^{d-1} - (a_{d-3}/d)\z^{d-2} - \left( (a_{d-4}/d) - (a^2_{d-2}/d^2) \right) \z^{d-3} + \cdots. 
\]
Using a similar technique, for each $m\geq 0$, corresponding to the biholomorphism $H:\Omega_{d^mc}' \rightarrow \Omega
_{d^{m+1}c}'$, we construct $\tilde{H}_{d^mc}: \mbb C \times \cal A_{d^mc} \rightarrow  \mbb C \times \cal A_{d^{m+1}c}$ which turns out to be of the same form as in (\ref{H tilde int}). These lifts are used to identify the automorphism
\[
\ga^{(m)}_{k/d^n}( z,\z) = \left( z + \frac{d}{a} \sum_{l = 0}^{\infty} \left(\frac{d}{a}\right)^l \left( Q\left(\z^{d^l}\right) - Q \left( \left(e^{2\pi ik/d^n} \cdot \z\right)^{d^l}\right)\right), e^{2 \pi i k/d^n} \cdot \z \right)
\]
of $ \mbb C \times \mathcal{A}_{d^mc} $ corresponding to any element $[k/d^n] \in \pi_1(\Om'_{d^mc})=\mathbb Z\left[1/d\right]$. Thus, we obtain a precise description of fibers of any point which plays the crucial role to deduce the form of $A$ and eventually the structure of \text{Aut}$(\Om'_c)$. Though this provides some understanding of ${\rm Aut}(\Om_c)$, we do not know its precise structure. On the other hand, there cannot be too many affine maps of $\mbb C^2$ contained in ${\rm Aut}(\Om_c)$.
\begin{thm}\label{linear group}
The space of affine maps of $\mbb C^2$ that preserve $\Om_c$ is a finite cyclic group. Furthermore, 
\begin{itemize}
\item[(i)] if $H(x,y)=(y, p(y)-ax)$ where $p(0)=a_0 \not= 0$, the only such map is the identity.
\item[(ii)] if $H(x,y)=(y, y^d-ax)$, $d \ge 2$, the affine maps that preserve $\Om_c$ are of the form $(\eta x, \eta^d y)$ where $\eta^{d^2-1}=1$, i.e., this group is isomorphic to $\mathbb{Z}_{d^2-1}$.
\end{itemize}
\end{thm}
Next we classify the possible biholomorphisms that can exist between a pair of punctured \text{Short}\;$\mbb C^2$'s that arise from the same H\'{e}non map. 

\begin{thm}\label{bihol short}
Let $H(x, y) = (y, p(y) - ax)$ be a H\'{e}non map as in (1.1). For $c_1, c_2 > 0$, let $\Om_{c_1}, \Om_{c_2}$ be a pair of \text{Short}\;$\mbb C^2$'s arising from $H$. Then  $\Om_{c_1}$  is biholomorphic to $\Om_{c_2}$ if and only if $c_1 = c_2 d^{\pm n}$, for some integer $n \ge 1$. 
\end{thm}

\begin{cor} \label{cor short}
There exists a continuum of pairwise non-biholomorphic Short $\mathbb{C}^2$'s with the same automorphism group.
\end{cor} 

\medskip 
\no 
\no {\bf Acknowledgement:} The authors would like to thank the referees for carefully reading the paper and making helpful comments.


\section{Proposition \ref{Prop 1} and some consequences}

\no The first thing to do is to show that punctured \text{Short}\;$\mbb C^2$'s are connected.

\medskip 

\no {\it Claim:} For every $c > 0$, $\Om'_c = \{0 < G_H^+ < c\}$ is connected.

\medskip 

\no The essential idea is to note that a given pair of points in $U^+$ can be pushed to $V_R^+$ by a large iterate of $H$. But some care must be exercised since working with $\Om_c$ means that we do not have the full freedom to move within $V^+$. To make this precise, pick points $A, B \in \Omega_c'$ and choose $c_1 > 0$ so that $\max \{G_H^+(A), G_H^+(B)\}< c_1 <c$.  For small $\epsilon>0$, there exists $R_{\epsilon} > 1$ such that 
\begin{equation}\label{estGH}
\log^+\lvert y \rvert- \epsilon < G_H^+(x,y) < \log^+\lvert y \rvert+ \epsilon 
\end{equation}
for all $(x,y)\in V_{R_\epsilon}^+$. Choose $n_0$ large enough such that $2\epsilon << d^{n_0}(c-c_1)$ and both $H^{n_0}(A)$, $H^{n_0}(B)$ are in $V_{R_\epsilon}^+$. Without loss of generality,  assume that 
\[
\left \lvert {(H^{n_0}(A))}_2\right\rvert \leq \left \lvert {(H^{n_0}(B))}_2 \right\rvert
\]
where $H^{n_0}(z)=\left({(H^{n_0}(z))}_1, {(H^{n_0}(z))}_2\right)$. Note that there exists a path  $\sigma \subseteq V_{R_\epsilon}^+$ joining $H^{n_0}(A)$ and $H^{n_0}(B)$  such that $\lvert y \rvert\leq \left \lvert {(H^{n_0}(b))}_2 \right\rvert$ for any point $(x,y)\in \sigma$. This means that the highest vertical displacement of $\sigma$ is no more than $\left \lvert {(H^{n_0}(b))}_2 \right\rvert$ at all times. Now
\[
G_H^+(x,y) < \log^+ \left \lvert {(H^{n_0}(B))}_2 \right\rvert+\epsilon < G_H^+(H^{n_0}(B))+2 \epsilon < d^{n_0}c
\]
for all $(x,y)\in \sigma$. It follows that $H^{-n_0}(\sigma) \subseteq \Omega_c'$ is a path which joins $a$ and $b$. Hence $\Omega_c'$ is connected and $\pa \Om'_c = \pa \Om_c \cup \partial K^+$. 

\medskip

\no The proof of Proposition \ref{Prop 1}  consists of the following steps:

\medskip	

\no (i)  For a given $b \in (0, c)$, there exists an $a \in (0, c)$ such that $f(\partial \Omega_b) = \partial \Omega_a$. To see this, by Theorem 7.2 of \cite{HOV}, $\partial \Omega_b$ is foliated by Riemann surfaces, each leaf of which is a copy of $\mbb C$ and is dense in $\partial \Omega_b$. Let $\cal{L}$ be a leaf of this foliation of $\partial \Omega_b$ and let $\psi: \mbb C \to \cal{L}$ be a parametrization. Then, $G_H^+\circ f \circ \psi$ is a bounded subharmonic function on $\mbb C$ and hence is constant. Thus, $(G_H^+\circ f)_{|\cal{L}}= a$ for some $a \ge 0$. Since $\cal L$ is dense in $\pa \Om_b$, it follows that $G_H^+ \circ f(\pa \Om_b) = a$, i.e., $f(\partial \Omega_b) \subset \Omega_a$. A similar argument for $f^{-1}$ gives $G_H^+ \circ f^{-1}(\partial \Omega_a)=b$, i.e., $f^{-1}(\partial \Omega_a) \subset \Omega_b$.
	
\medskip

\no (ii) Suppose $a=0$. Then $f(\partial \Omega_b) \subset K^+$. The claim now is that either $f(\pa \Om_b) \subset J^+$ or $f(\pa \Om_b) \subset {\sf int}(K^+)$. Indeed, suppose that $f(\pa \Om_b) \cap C \not= \emptyset$ for some component $C \subset {\sf int}(K^+)$. Then there exists a $z_0 \in \pa \Om_b$ and a neighbourhood $N$ of it that is mapped by $f$ into $C$. Since $G_H^+ : U^+ \ra (0, \infty)$ is a submersion, $N$ can be written as the union of the smooth, real analytic hypersurfaces $\{ G_H^+ = \al\} \cap N$ where $\al$ varies in an open interval 
$(b', b'')$ with $b' < b < b''$. By working with a dense leaf in each of the hypersurfaces $\{ G_H^+ = \al\}$, $b' < \al < b''$, the same reasoning as in (i) shows that $f(\pa \Om_{\al}) \in K^+$ for all $\al \in (b', b'')$. Therefore, $f$ maps the entire tubular neighbourhood $\Om_{b', b''} = \Om_{b''} \sm \ov \Om_{b'}$ into $K^+$. Since $f \in {\rm Aut}(\Om_c)$, it follows that $f(\Om_{b', b''}) \subset C$. In particular, $f(\pa \Om_b) \subset C$ and this completes the proof of the claim.

\medskip

\no (iii) To proceed, we now claim that $f(J^+) = J^+$ and $f(K^+) =  K^+$. To see this, suppose that $f(w) \in J^+$ for some $w \in \Om'_c = \Om_c \sm K^+$. Consider the level set $\{G_H^+ = \beta\}$, $\be > 0$, that contains $w$. By (i) and (ii), $f(\pa \Om_{\be}) \subset J^+$. By \cite{BS1} it is possible to choose a periodic saddle point $p$ for $H$ such that the corresponding stable manifold $W^s(p) \simeq \mbb C$ and whose closure is $J^+$. By the reasoning in (i), $G_H^+ \circ f^{-1}$ is constant on $W^s(p)$ and hence $f^{-1}(J^+) \subset \{G_H^+ = \delta\}$ for some $\de \ge 0$. As $f(\pa \Om_{\be}) \subset J^+$, $\de$ must equal $\be$. Thus, $f(\pa \Om_{\be}) = J^+$ and this shows that $J^+$ is a smooth manifold. But $J^+$ can never be smooth by \cite{BK}. Hence $f^{-1}(J^+) \subset K^+$.

\medskip

If $K^+$ has no interior, this shows that $f^{-1}(J^+) \subset J^+$ and similarly, $f(J^+) \subset J^+$. Thus, $f(J^+) = J^+$. If ${\sf int}(K^+) \not= \emptyset$, $G^+ \circ f^{-1}$ is  plurisubharmonic on each component of ${\sf int}(K^+)$ and vanishes on $\pa K^+ = J^+$ since $f^{-1}(J^+) \subset K^+$. By the maximum principle, $G_H^+ \circ f^{-1} < 0$ on ${\sf int}(K^+)$. This means that $f^{-1}({\sf int}(K^+)) \subset {\sf int}(K^+)$. Similar arguments show that $f({\sf int}(K^+)) \subset {\sf int}(K^+)$. Hence $f({\sf int}(K^+)) = {\sf int}(K^+)$ and $f(J^+) = J^+$. Thus in both cases $f(K^+) = K^+$ and $f(J^+) = J^+$.

\medskip

\no (iv) By (iii), it follows that $a > 0$ and hence $f (\pa \Om_b) \subset \pa \Om_a$ and similar arguments applied to $f^{-1}$ show that $f(\pa \Om_b) = \pa \Om_a$. Further, since $f$ preserves $K^+$, we also conclude that $f : \Om_b \ra \Om_a$ is a biholomorphism.

\medskip

\no (v) To show that $a = b$, note that both $\pm(b^{-1}aG_H^+ - G_H^+ \circ f)$ are pluriharmonic on $\Om'_b = \Om_b \sm K^+$ since $f$ preserves
 $K^+$. Both vanish on $\pa \Om'_b$ and hence the maximum principle shows that
\begin{equation}\label{e:Green}
b^{-1}a G_H^+ = G_H^+ \circ f
\end{equation}
on $\Om'_b$. If $b^{-1}a > 1$, choose $ \ti{c}>0$ such that $\ti c b^{-1}a >c > \ti c$. Now (\ref{e:Green})  and observation (i), shows that $f(\pa \Om_{\ti c})=\partial \Om_{\ti cb^{-1}a} \nsubseteq \Om_c$ and this is a contradiction. Thus, $b^{-1}a \le 1$. A similar argument for $f^{-1}$ shows that $b^{-1}a \ge 1$ and hence $a = b$. This completes the proof of Proposition \ref{Prop 1}.

 \begin{cor}
 Suppose $f \in  {\rm Aut}(\mathbb{C}^2) \cap {\rm Aut}(\Omega_{c}).$ Then $f \in {\rm Aut}(\Omega_{d})$ for every $d \ge 0.$
\end{cor}
\begin{proof}
For $0 \le d<c$, this follows from Proposition \ref{Prop 1}. For $d \ge c$, note that
\[
h=G_H^+ - G_H^+ \circ f 
\]
 is pluriharmonic on $\mbb C^2 \setminus K^+$ and vanishes on the open set $\Om'_c = \Omega_{c} \setminus K^+$. Thus, $h \equiv 0$ on $\mbb C^2 \setminus K^+$ which completes the proof. 
\end{proof}

\begin{cor}\label{Omega cd}
Let $c,d>0$ and $c \neq d$ be such that $\Omega_c$ is biholomorphic to $\Om_d$. Then any biholomorphism $\phi: \Omega_c \to \Omega_d$ preserves $K^+$, i.e., $\phi(K^+)=K^+$.
\end{cor}
\begin{proof}
By using similar arguments as in the proof of Proposition \ref{Prop 1}, 
it follows that $G_H^+ \circ \phi$ is constant on $J^+$. Thus $\phi(J^+) \subset \partial \Omega_a$ where $0< a < d$ or $\phi(J^+) \subset K^+$.
If $\phi(J^+) \subset \partial \Omega_a$ for some $0<a<d$, then as before, by using the fact that $\partial \Omega_a$ is foliated by
Riemann surfaces with each leaf biholomorphic to $\mathbb{C}$ and dense in $\partial \Omega_a$, it must be the case that 
$\phi(J^+)=\partial \Omega_a$, i.e., $J^+$ is a smooth manifold and this is a contradiction as \cite{BK} shows. Further, by the maximum principle it follows that $\phi(K^+) \subset K^+$. Now, applying the same reasoning to $\phi^{-1}$ gives $\phi^{-1}(K^+) \subset K^+$ and hence $\phi(K^+)=K^+$.  
\end{proof}

Proposition \ref{Prop 1} says that an automorphism of $\Om_c$ is an automorphism of every smaller $\Om_b$, $b < c$. The same property holds for automorphisms of $\Om'_c$. As mentioned, we do not know if every automorphism of $\Om'_c$ extends to an automorphism of $\Om_c$.  

\begin{cor} 
		Let $\Om_c'=\Om_c \setminus K^+$ be a punctured  \text{Short}\;$\mbb C^2$ and let $f \in {\rm Aut}(\Om_c')$. Then $f \in {\rm Aut}(\Om'_b)$ for every $0<b<c$.
		\end{cor}
		\begin{proof}
		Define the function $\tilde{G}_H$ on $\Om_c' \cup \textsf{int}(K^+)$ as follows:
		\begin{itemize}
	\item $\tilde{G}_H(z)=G_H^+\circ f(z)$ for $z \in \Om_c'.$
	\item $\tilde{G}_H(z)=G_H^+ \equiv 0$ for $z \in \textsf{int}(K^+).$
		\end{itemize}
		Note that $\tilde{G}_H$ is pluriharmonic on $\Om_c\setminus J^+$. Let $\tilde{G}_H^*$ be the upper semicontinuous regularization of $\tilde{G}_H$ on $\Om_c.$ Thus, $\tilde{G}_H^*$ is non-constant plurisubharmonic on $\Omega_c$ and $\tilde{G}^*_H<c$ on $\Om_c.$ Since $\tilde{G}_H$ is continuous on $\Om_c' \cup \textsf{int}(K^+)$, which is an open subset of $\mathbb{C}^2$, it follows that $\tilde{G}_H=\tilde{G}_H^*$ on $\Om_c' \cup \textsf{int}(K^+)$. 

\medskip

\no {\it Claim:} If $\tilde{G}_H^*(z) \neq 0$ for some $z \in J^+$, then $\tilde{G}_H^*(z)=c$.

\medskip 
		
	\no If not, there exists a $z_0 \in J^+$ such that $\tilde{G}_H^*(z_0)=a$ for some $0<a<c$. There exists a sequence $\{z_n\} \in \Om_c'$ with $z_n \to z_0$ and $\tilde{G}_H(z_n) \to a$. Let $\tilde{G}_H(z_n)=a_n$ and $G_H^+(z_n)=b_n$. Note that $a_n \to a$ and $b_n \to 0 $ as $n \to \infty$. By (i) in the proof of Proposition \ref{Prop 1}, 
\begin{equation}
		f(\partial \Om_{b_n})=\partial \Om_{a_n} \text{ and }\tilde{G}_H(\partial \Om_{b_n})=a_n.
\end{equation}
Choose $w_0 \in \pa \Om_a \subset \Om'_c$ and a sequence $\{w_n\} \subset \Om'_c$ such that $w_n \to w_0$ and $w_n \in \{G_H^+ = a_n\}$. By (2.2), $f^{-1}(w_n) \in \pa \Om_{b_n}$. Since $b_n \to 0$, either $\vert f^{-1}(w_n) \vert \ra +\infty$ or $f^{-1}(w_n)$ clusters at a point on $J^+$. But none of these possibilities can hold since $f$ is an automorphism of $\Om'_c$.

\medskip

Finally, as $\tilde{G}_H^*$ is a non-constant bounded plurisubharmonic function on $\Om_c$, with upperbound $c$, the maximum principle shows that $\tilde{G}_H^*(z)=0$ for every 
$z \in J^+$, i.e., $\tilde{G}_H$ extends as a continuous plurisubharmonic function on $\Om_c$ such that 
${\tilde{G}_H}\equiv 0$ on $K^+$.
Since both $G_H^+$ and $\tilde{G}_H$ is pluriharmonic on $\Om_c'$, the function $G_H^+(z)-\tilde{G}_H(z)$ is pluriharmonic on $\Om_c'$ and identically vanishes on the boundary of $\Om_c'$. Hence by the maximum principle of harmonic functions $G_H^+=\tilde{G}_H$ on $\Om_c'$. Also, as both $\tilde{G}_H=G_H^+\equiv 0$ on $K^+$, the proof follows.
\end{proof}

\section{Proof of Theorem \ref{linear group}}

Note that by Proposition 1.2, every automorphism of $\Om_c$ preserves $K^+$. Thus, we begin by characterizing those affine maps that preserve $K^+$. 

\begin{prop}\label{p:linear}
Suppose $L$ is an affine map of $\mathbb{C}^2$ such that $L(K^+)=K^+$. Then $L(K^-)=K^-$ and there exist constants $e, f$ with $|e|=|f|=1$ 
such that $L(x,y)=(ex,fy).$ Further 
\[H \circ L \circ H^{-1}=H^{-1} \circ L \circ H, \text{ i.e., } L \circ H^2=H^2 \circ L.\]
\end{prop}
\begin{proof}
Let $L(x,y)=(a_{1} x+a_2 y+a_3,b_1 x+b_2 y+b_3)$.	
Consider a sequence of points $[x_n:y_n:1]\in \mathbb{P}^2$ such that $(x_n,y_n)\in K^+$ is an unbounded sequence  and $[x_n:y_n:1]\rightarrow [1:0:
0]$ as $n\rightarrow \infty$. Hence,  
\[
\frac{y_n}{x_n}\rightarrow 0
\] 
as $n\rightarrow \infty$. Since $L(K^+)=K^+$ and $L(x_n,y_n)\in K^+$ is also an unbounded sequence, 
\[
[a_1 x_n+ a_2 y_n+a_3: b_1 x_n+ b_2 y_n+b_3 : 1]\rightarrow [1:0:0]
\]
in $\mathbb{P}^2$. Therefore $b_1=0$ which means that $L(x,y)=(a_1 x+a_2 y+a_3,b_2 y+b_3)$.

\medskip\no Let $R=H \circ L.$ Since $H(x,y)=(y,p(y)-ax)$,
\begin{align*}
R(x,y)=( b_2 y+  b_3,p(b_2y+b_3)-a (a_1 x+a_2 y+a_3))
\end{align*}
is a regular map with $R(K^+)=K^+$ and the indeterminacy sets of $R$ are $I_R^+=[1:0:0]$ and $I_R^-=[0:1:0].$ Note that 
the degree of $R$ is equal to the degree of $H$. Hence from Theorem 5.4 
of \cite{L}, there exists an integer $n \ge 1$ such that $R^n=H^n$. This shows that $R(K^\pm)=K^\pm$ and $G_R^\pm=G_H^\pm$. This also proves that $L(K^-)=K^-$ and by Theorem 1.1 of \cite{BPV:Rigidity}, we have $a_2=0.$ 

\medskip
Further, by Theorem 1.1 from \cite{BPV:Rigidity}, there exist linear maps $C_{\eta_1}(x,y)=(\eta_1 x,\eta_1^{-1} y)$ and $C_{\eta_2}(x,
y)=(\eta_2 x,\eta_2^{-1} y)$ where $|\eta_i|=1$ for $i=1,2$, such that
\[
H\circ L \circ H(x,y)= C_{\eta_1} \circ L\circ H^2(x,y) \text{ and } H^{-1} \circ L \circ H^{-1} (x,y)=C_{\eta_2} \circ L\circ H^{-2}(x,y),
\]
i.e.,
\begin{align}\label{e: linear 1}
H\circ L \circ H^{-1}(x,y)= C_{\eta_1} \circ L (x,y) \text{ and } H^{-1} \circ L \circ H (x,y)=C_{\eta_2} \circ L(x,y)
\end{align}
for every $(x,y) \in \mathbb{C}^2.$ Recall that $H^{-1}(x,y)=(a^{-1}(p(x)-y), x).$ Thus by equating the components on both sides of (\ref
{e: linear 1}), the following relations hold:
\begin{align*}
\eta_1 a_1 x&=b_2 x+b_3, \\ \eta_2^{-1}b_2 y&= a_1 y+a_3, \\
\eta_1^{-1}b_2y&=p(b_2x+b_3)-a_1(p(x))+a_1y-aa_3.
\end{align*}
Finally, by comparing the coefficients of the monomials on both sides of the above equation, we conclude that $L(x,y)=(a_1 x,b_2 y)$ 
where $|a_1|=|b_2|=1$ and $\eta_1=\eta_2.$
\end{proof}
Next, we use the above results appropriately to complete the proof of Theorem \ref{linear group}

\begin{proof}[Proof of Theorem \ref{linear group}]
Let $L$ be an affine map preserving $\Om_c$. Then by Proposition \ref{Prop 1}, $L$ is an affine map that preserves $K^+.$ Further, by Proposition \ref{p:linear},
$ H^2 \circ L=L \circ H^2$ where $L(x,y)=(ex,fy)$. Now,
by comparing components, we obtain
\begin{align*}
p(fy)-aex&=ep(y)-aex, \\
 p\big((p(fy)-aex)\big)-afy&=fp\big((p(y)-ax)\big)-afy.
\end{align*}
Thus 
\begin{align}\label{e: linear2}
p(fz)=ep(z) \text{ and } p(ez)=fp(z)
\end{align}
for every $z \in \mbb C$, and by comparing the highest degree terms in (\ref{e: linear2}), we get $f=e^{d}$ and $e=f^{d}$. In other words, $e^{d^2-1}=f^{d^
2-1}=1$ where $d$ is the degree of the H\'{e}non map $H$. 
Hence, all possible choices for $e$ form a subset of the $d^2-1$ roots of unity which shows that the affine
maps of $\mathbb{C}^2$ preserving $\Om_c$, and consequently preserving $K^+$, is a subgroup of the 
finite cyclic group $\mathbb{Z}_{d^2-1}.$

\medskip Now, if $p(0)=a_0 \neq 0$, (\ref{e: linear2}) is satisfied if and only if $e=f=1$, i.e., $L=\text{Identity}.$ Finally, for $p(y)=y^d$, (\ref{e: linear2}) gives $e^d=f$ and $f^d=e$, i.e., $e^{d^2-1}=1$ and $L(x,y)=(ex,e^d y)$. By writing $L_\eta=(\eta x, \eta^d y)$ where $\eta$ is a $(d^2-1)$-th root of unity, we note that $H \circ L_\eta=L_{\eta^d} \circ H$. This implies that $G_H^+(x,y)=G_H^+\circ L_\eta (x,y)$ for $(x,y) \in \mathbb{C}^2$, whenever $\eta^{d^2-1}=1.$ Thus, the group of affine maps that preserve $\Om_c$ is isomorphic to 
$\mathbb{Z}_{d^2-1}$. This completes the proof of Theorem \ref{linear group}.
\end{proof}  
\section{The Fundamental group of $\Omega_c'$} 

Recall that $\mathbb{Z}[1/d]=\left\{{m}/{d^n}:m,n\in \mathbb{Z}\right\}$ consists of all rational numbers whose denominator is an integral power of $d$.
\begin{prop} \label{fundamental}
For any  $c>0$, the fundamental group of $\Omega_c'$ is $\mathbb{Z}[1/d]$. 
\end{prop}

\begin{proof}
Let $\phi^+$ be the B\"{o}ttcher coordinate of $H$ in $V_R^+$ and $\omega={d\phi^+}/{\phi^+}$. By Proposition 7.3.2 in \cite{MNTU},  $\omega$ is a closed 1-form on $I^+$ and $H^*(\omega)=d \omega$.  For any closed curve $C$ in $\Omega_c'$, set
\[
\alpha(C)=\frac{1}{2\pi i} \int_C \omega.
\]
Since there exists $n_0\geq 0$ such that $H^{n_0}(C) \subseteq V_R^+$,
\begin{equation} \label{winding}
\alpha(H^{n_0}(C))=\frac{1}{2\pi i} \int_{H^{n_0}(C)} \omega=\frac{1}{2\pi i}\int_C H^{n_0*} \omega =\frac{1}{2\pi i}\int_C d^{n_0} \omega=d^{n_0} \alpha(C).
\end{equation}
Now $\alpha(H^{n_0}(C))$ is the winding number of the curve $H^{n_0}(C) \subset V_R^+$ around the $x$-axis and hence $\alpha(C)\in \mathbb{Z}[{1}/{d}]$. 

\medskip 
 
 Recall that for any $\epsilon>0$, there exists $R_\epsilon >R>0$ such that (\ref{estGH}) holds, i.e.,
\begin{equation*}
\log \lvert y \rvert-\epsilon <G_H^+(x,y) < \log \lvert y \rvert+\epsilon,
\end{equation*} 
for $\lvert y \rvert>R_\epsilon$. Let $\epsilon<<1$ and choose $c> R_\epsilon+2$. Then  
\[
B_{\epsilon,c}=V_{\left(R_\epsilon+\frac{1}{2}\right)}^+ \bigcap\left \{\lvert x \rvert, \lvert y \rvert\leq R_\epsilon+1\right \}\subseteq \Omega_c'.
\] 
Let 
\[
C_0:  t \mapsto \left(0, R' e^{2\pi i t}\right) 
\]
be a closed curve in $B_{\epsilon,c}\subseteq \Omega_c'$, where $R_\epsilon + 1/2<R'< R_\epsilon+1$. Clearly, $\alpha(C_0)=1$, and consequently $\alpha(m C_0)=m$ for any $m\in \mathbb{Z}$. This implies that for any $m\in \mathbb{Z}$, there exists a closed curve $C_m\subseteq \Omega_c'$ such that $\alpha(C_m)=m$.  Since (\ref{winding}) holds and $H^{-n}(\Omega_c')\subseteq \Omega_c'$ for any $n \ge 1$, it follows that $H^{-n}(C_m)\subseteq \Omega_c'$ with  $\alpha(H^{-n}(C_m))=m/d^n$.  Thus, $\alpha$ is a surjection from 
$\pi_1(\Omega_c')$ onto $\mathbb{Z}[{1}/{d}]$.

\medskip 

To show that $\alpha$ is injective, let $C$ be any closed curve in $\Omega_c'$ such that  $\alpha(C)=0$. Then there exist $c_1, c_2$ with $0<c_1<c_2<c$ such that 
\[
C\subseteq \left\{(x,y)\in \mathbb{C}^2: c_1< G_H^+(x,y) <c_2\right\}.
\]
Let $\epsilon>0$. Take $n_0$ large enough such that $d^{n_0}(c-c_2)>>\epsilon$ and
 such that 
\[
H^{n_0}(C)\subseteq V_{R_\epsilon}^+ \bigcap \left \{(x,y)\in \mathbb{C}^2: d^{n_0} c_1< G_H^+(x,y) <d^{n_0} c_2\right\}
\]
with $R_\epsilon>R$. Also since $H^{n_0}(C)$ is compact, there exists $R'>0$ such that 
\[
H^{n_0}(C) \subseteq V_{R_\epsilon}^+ \bigcap \left\{(x,y)\in \mathbb{C}^2: \lvert x \rvert, \lvert y \rvert \leq R' \right\}.
\]
Let 
\[
R_{s}= \sup\{\lvert y_0 \rvert: \text{ the complex line } y=y_0\text{ intersects } H^{n_0}(C)\}
\]
and  
\[
R_{i}= \inf\{\lvert y_0 \rvert: \text{ the complex line } y=y_0\text{ intersects } H^{n_0}(C)\}.
\]  
\medskip 

\no 
{\it Claim:} $V_{R_\epsilon}^+ \cap \{(x,y): R_i \leq \lvert y\rvert \leq R_s\} \subseteq \Omega_{d^{n_0} c}'$.

\medskip
 Let $(\tilde{x},\tilde{y})\in V_{R_\epsilon}^+ \cap \{(x,y): R_i \leq \lvert y\rvert \leq R_s\}$ be such that $G_H^+(\tilde{x},\tilde{y})\geq d^{n_0}c$. Then there exists a point $(x_c,y_c)\in H^{n_0}(C)$  such that $\lvert y_c \rvert=\lvert \tilde{y}\rvert$. Note that $G_H^+(x_c,y_c)< d^{n_0} c_2$. Since $d^{n_0}(c-c_2)>> \epsilon$, this is a contradiction. 
 
 \medskip 
  
 Now since $\alpha(C)=0$, it follows that $\alpha(H^{n_0}(C))=0$. This implies that $H^{n_0}(C)$ is null-homotopic in $V_R^+$. In fact, $H^{n_0}(C)$ is null-homotopic in 
 \[
 V_{R_\epsilon}^+ \cap \{(x,y): R_i-\delta < \lvert y\rvert < R_s+\delta\} \subseteq \Omega_{d^{n_0} c}' 
 \] 
  for $\delta>0$ small enough, which in turn shows that $C$ is null-homotopic in $\Omega_c'$. Therefore,
 for $c>0$ sufficiently large,  there is a bijective correspondence from the set of all closed paths in $\Omega_c'$ and $\mathbb{Z}[{1}/{d}]$. 



\medskip 

Let $p\in \Omega_c'$. We will prove that $\pi_1(\Omega_c',p)=\mathbb{Z}[{1}/{d}]$. For a given $g_p \in \mbb Z[1/d]$, there exists a closed curve $C\in \Omega_c'$ such that 
$\alpha(C)=g_p$.  Let $C$ be parametrized by $\beta_q$ with $\beta_q(0)=q$. Join $p$ and $q$ by a curve $\gamma_q \subseteq \Omega_c'$ (which is connected). The closed curve $C'=\gamma_q \cup \beta_q \cup (-\gamma_q)$ has base point at $p$ and  $\alpha(C')=g_p$.  Thus, for all $c>0$ sufficiently large, $\pi_1(\Omega_c')=\mathbb{Z}[{1}/{d}]$. Now since for any $c>0$, there exists a positive integer $n_c \ge 1$ sufficiently large such that $H^{n_c}(\Omega_c')=\Omega_{d_{n_c} c}'$,  we get that $\pi_1(\Omega_c')= \mathbb{Z}[{1}/{d}]$ for all $c>0$. Since $\mathbb{Z}[{1}/{d}]$ is commutative, the fundamental group of $\Omega_c'$ is independent of the base point.
\end{proof}


\section{Construction of a covering of $\Omega_c'$}
\no 
Let $c>0$ be large enough such that $\Omega_c' \cap V_R^+\neq \emptyset$. 
Fix a base point $a\in \Omega_c' \cap V_R^+ $ and consider the set of all pairs $(z,C)$ where $C$ is a path in $\Omega_c'$ joining $a$ and $z$. Introduce the equivalence relation
\[
(z,C) \sim (z',C')
\] 
if and only if $z=z'$ and $[CC'^{-1}]\in \mathbb{Z} \subseteq \pi_1(\Omega_c')$. Let $\widehat {\Omega}_{c}'$ denote the set of all equivalence classes and let $\widehat{\pi}_c:(z, C)\mapsto z$ be the natural projection map from $\widehat {\Omega}_{c}'$ to  $\Omega_c'$. Equip $\widehat {\Omega}_{c}'$ with the pull-back complex structure  so that the projection map $\widehat{\pi}_c$ becomes holomorphic. The assignment $\widehat{\phi}^+_c:\widehat {\Omega}_c'\rightarrow \mathbb{C}$ given by
\begin{equation}\label{defn phi hat}
\widehat{\phi}^+_c\left([z,C]\right)=\phi^+(a) \exp \int_C \omega,
\end{equation}
defines a holomorphic function, where $\phi^+$ is the B\"{o}ttcher function on $V_R^+$ and $\omega$ is as in Proposition\  \ref{fundamental}.  Now since $[z,C]=[z,C']$ implies $[CC'^{-1}]\in \mathbb{Z}$, it follows that  $\widehat{\phi}^+_c$ is well-defined. Also note that if $C$ lies completely in $V_R^+$, then  
\begin{equation}\label{PhiV}
\widehat{\phi}^+_c\left([z,C]\right)=\phi^+(a) \exp \int_C \omega=\phi^+(a) \exp \int_C \frac{d\phi^+}{\phi^+}=\phi^+\left(\pi_c[z,C]\right)=\phi^+(z).
\end{equation}
The goal of this section is to show that $\widehat {\Omega}_c'$ is biholomorphic to $\mathbb{C}\times \mathcal{A}_c$ where 
\[
\mathcal{A}_c=\{z\in \mathbb{C}: 1< \lvert z\rvert< e^c\}.
\]
As will be apparent shortly, this will be done by repeatedly using $H^n(\Omega_c')=\Omega_{d^nc}'$.

\medskip 
\no 
{\it A chain of covering maps:} As above, it is possible to construct a covering $\widehat {\Omega}_{d^nc}'$ of $\Omega_{d^n c}'$ with base point $H^n(a)\in \Omega_{d^n c}' \cap V_R^+ $ for all $ n \ge 1$. Define $\widehat{H}_c: \widehat{\Omega}_{c}' \rightarrow\widehat{ \Omega}_{d c}' $ by
\[
[z,C] \mapsto [H(z), H(C)]
\]
and similarly 
$\widehat{H}_{d^nc}: \widehat{\Omega}_{d^nc}' \rightarrow\widehat{ \Omega}_{d^{n+1} c}' $, for all $n \ge 1$. Clearly, the $\widehat{H}_{d^nc}$'s are well-defined and this leads to the chain  
 \[
 \begin{CD}
 \widehat{\Omega}_c' @>{\widehat{H}}_c >>  \widehat{\Omega}_{dc}'@>{\widehat{H}}_{dc} >> \widehat{\Omega}_{d^2 c}' @>{\widehat{H}}_{d^2c}>>  \widehat{\Omega}_{d^3 c}' @>  {\widehat{H}}_{d^3 c} >> \cdots.
 \end{CD}
 \]
 Define $\widehat{H}_c^n: \widehat{\Omega}_{c}' \rightarrow\widehat{ \Omega}_{d^n c}' $ by
\[
[z,C]\in \widehat{\Omega}_c' \mapsto [H^n(z), H^n(C)].
\]
Clearly, 
\[
\widehat{H}_c^n= \widehat{H}_{d^{n-1} c}\circ\cdots \circ \widehat{H}_c.
\]

\medskip
\no 
{\it Functorial property of $\widehat{\phi}^+_{d^nc}$'s:} As in (\ref{defn phi hat}), one can construct the B\"{o}ttcher-like maps $\widehat{\phi}^+_{d^nc}$ for all $n \ge 1$.
Further, since 
\[
\widehat{\phi}_{dc}^+ [H(z), H(C)]=\phi^+ (H(a))\exp \int_{H(C)}\omega,
\]
it follows from $H^*(\omega)=d\omega$ that
\begin{equation} \label{funct}
\widehat{\phi}^+_{dc} \circ \widehat{H}_c = {\left(\widehat{\phi}^+_c\right)}^d
\end{equation}
and in general, 
\begin{equation}\label{funct1}
\widehat{\phi}^+_{d^n c} \circ \widehat{H}_{d^{n-1} c} \circ \cdots \circ  \widehat{H}_c= {\left(\widehat{\phi}^+_c\right)}^{d^n},
\end{equation}
for all $n$.

\medskip 

For all $n\geq 1$, define
\[
\widehat{\Omega_{d^n c}' \cap V_R^+}=\left\{[z,C]\in \widehat{\Omega}_{d^n c}': C \text{ is a path in } V_R^+\right\}.
\]
The following proposition is immediate.
\begin{prop}\label{V}
For $c>0$ sufficiently large,  $\bigcup_{n=0}^\infty {(\widehat{H}_c^n)}^{-1}(\widehat{ \Omega_{d^n c}'\cap V_R^+})=\widehat{\Omega}_c'$.
\end{prop}
Now we paraphrase \cite[Lemma 7.3.7]{MNTU} (also see \cite [Prop.\ 2.2] {Favre}).
\begin{itemize}
\item 
For sufficiently large $M>0$, let 
\[
U_R^+=\left\{(x,y)\in V_R^+: \lvert \phi^+(x,y)\rvert > M \max\{R, \lvert x\rvert\}\right\},
\]
where $\phi^+$ is the B\"{o}ttcher coordinate of $H$ in $V_R^+$. One can check that $H(U_R^+)\subseteq U_R^+$. There exists  a holomorphic function $\psi$ on $U_R^+$ such that
\[
\psi (H(z))=\frac{a}{d}\psi(z)+ Q(\phi^+(z))
\]
for all $z\in U_R^+$, where $Q$ is a polynomial of degree $d+1$ of a single variable. 

\item Define a map $\Phi: U_R^+ \rightarrow \mathbb{C}^2$ as follows: $(x,y)\mapsto (\psi(x,y),\phi^+(x,y))$.  One can prove that $\Phi$ is injective and for  $\rho>1$ sufficiently large,
\begin{equation} \label{STrho}
\{(s,t)\in \mathbb{C}^2: \lvert t \rvert >M, \lvert s \rvert <\rho  {\lvert t \rvert}^d \} \subseteq \Phi(U_R^+).
\end{equation}
\end{itemize}
For $n\geq 1$ and  $R>0$ sufficiently large, define 
\[
\widehat{\Omega_{d^n c}' \cap U_R^+}=\left\{[z,C]\in \widehat{ \Omega_{d^n c}' \cap V_R^+ }: \hat{\pi}_{d^n c}[z,C]\in U_R^+\right \}.
\]
\begin{prop}\label{UOmega}
Let $R>0$ be sufficiently large.  Then there exists a $c_0>0$ such that 
\begin{equation}\label{UOmega c}
\bigcup_{n=0}^\infty {\left(\widehat{H}_c^n\right)}^{-1}\left(\widehat{ \Omega_{d^n c}'\cap U_R^+}\right)=\widehat{\Omega}_c'
\end{equation}
for all $c \geq c_0$.
\end{prop}
\begin{proof}
Since $H(U_R^+)\subseteq U_R^+$, it follows that 
$
\widehat{H}_c(\widehat{\Omega_{c}' \cap U_R^+})\subseteq \widehat{\Omega_{d c}' \cap U_R^+}.
$
Let $[z,C]\in \widehat{\Omega}_c'$, then clearly,
$
\widehat{H}_c^n[z,C]\in \widehat{\Omega_{d^n c}'\cap U_R^+ },
$
for some $n$. The other inclusion is evident. 
\end{proof}

\begin{lem} \label{Phi H G Phi}
There exist holomorphic functions $\widehat{\psi}_c$ and $\widehat{\psi}_{dc}$ on $\widehat{\Omega}_c'$ and $\widehat{\Omega}_{dc}'$, respectively such that 
\begin{equation}\label{psiphi1}
\widehat{\psi}_{dc}\circ \widehat{H}_c\left([z,C]\right)=\frac{a}{d} \left( \widehat{\psi}_{c}\left([z,C]\right)\right)+ Q\left(\widehat{\phi}^+_c([z,C])\right),
\end{equation}
for all $[z,C]\in \widehat{\Omega}_c'$. More generally,  for all $n\geq 1$, there exists  a holomorphic function $\widehat{\psi}_{d^nc}$ on $\widehat{\Omega}_{d^n c}'$ such that 
\begin{equation}\label{psiphi2}
\widehat{\psi}_{d^nc}\circ \widehat{H}_{d^{n-1}c}([z,C])=\frac{a}{d} \left(\widehat{\psi}_{d^{n-1}c}([z,C])\right)+ Q\left(\widehat{\phi}^+_{d^{n-1}c}\left([z,C]\right)\right),
\end{equation}
for all $[z,C]\in \widehat{\Omega}_{d^{n-1}c}'$.
\end{lem}
\begin{proof}
For $[z,C]\in \widehat{\Omega_c' \cap U_R^+ }$, let
\[
\widehat{\psi}_c [z,C]=\psi(z).
\]	
Inductively, we can define $\widehat{\psi}_c$ on $\widehat{\Omega}_c'$ as follows. Let $[z,C]\in \widehat{\Omega}_c'$ and let $\widehat{H}_c^n[z,C]\in \widehat{ \Omega_{d^n c}' \cap U_R^+}$, for some $n\in \mathbb{N}$. Then, set
\[
\widehat{\psi_c}[z,C]=\frac{d^n}{a^n} \left(\widehat{\psi}_{d^n c} \circ \widehat{H}_c^n[z,C]\right)-\frac{d^n}{a^n} \left(Q\left( \widehat{\phi}^+_{d^{n-1} c} \circ \widehat{H}_c^{n-1}[z,C]\right)\right)-\cdots- \frac{d}{a} Q\left( \widehat{\phi}^+_c[z,C]\right).
\]
Similarly, we define $\widehat{\psi}_{d^nc}$ on $\widehat{\Omega}_{d^n c}'$, for each $n\geq 1$. Clearly,  (\ref{psiphi1}) and (\ref{psiphi2}) follow. 
\end{proof}

\no 
{\it Construction of a biholomorphism between $\widehat{{\Omega}}_c'$ and $\mathbb{C}\times \mathcal{A}_c$}:
Let $[z,C]\in \widehat{\Omega}_c'$. By Proposition \; \ref{UOmega}, it follows that $\widehat{H}_c^n[z,C]\in \widehat{ \Omega_{d^n c}'\cap U_R^+ }$, for some $n \ge 1$. This implies that 
\[
[H^n(z), H^n C] \in \widehat{ \Omega_{d^n c}'\cap U_R^+ }.
\]
Further, by (\ref{funct1}), we have
\[
\widehat{\phi}^+_{d^n c}[H^n z, H^n C]={\left( \widehat{\phi}^+_c [z,C]\right)}^{d^n}=\phi^+ \circ H^n (z).
\]
Now since
\begin{equation}
G^+=\log \lvert \phi^+\rvert
\end{equation}
on $V_R^+$ and $H^n(z)\in V_R^+$, it follows that 
\[
0<\log \left(\left \lvert \phi^+ \circ H^n(z) \right\rvert \right) <d^n c.
\]
Therefore, 
\[
1<{\left\lvert\widehat{\phi}^+_c[z,C]\right\rvert}^{d^n} <\exp (d^n c),
\]
for all $n$ sufficiently large, which in turn gives 
\[
1<\left\lvert \widehat{\phi}^+_c[z,C]\ \right\rvert < \exp (c),
\]
for all $[z,C]\in \widehat{\Omega}_c'$. Therefore, the image of $\widehat{\Omega}_c'$ under the map $\widehat{\phi}^+_c$ is the annulus $\mathcal{A}_c$ as defined earlier.

 \medskip 
 
Define $\widehat{\Phi}_c : \widehat{\Omega}_c'  \rightarrow \mathbb{C}\times \mathcal{A}_c$  by
\[
\widehat{\Phi}_c([z,C]) = \left(\widehat{\psi}_c([z,C]), \widehat{\phi}_c^+([z,C])\right).
\] 
and $G_c : \mathbb{C}\times \mathcal{A}_c \rightarrow \mathbb{C}\times \mathcal{A}_{dc}$ by
\begin{equation} \label{lift}
G_c(s,t):= \left(\frac{a}{d}s+Q(t), t^d\right).
\end{equation}
Similarly, one can define $\widehat{\Phi}_{d^n c}: \widehat{\Omega}_{d^n c}' \rightarrow \mathbb{C}\times \mathcal{A}_{d^n c}$ and $G_{d^n c}: \mathbb{C}\times \mathcal{A}_{d^n c} \rightarrow \mathbb{C}\times \mathcal{A}_{d^{n+1}c}$, for all $n\geq 1$. Also, note that by Lemma \ref{Phi H G Phi} we have

\begin{equation} \label{cover}
\widehat{\Phi}_{d^n c} \circ \widehat{H}_{d^{n-1}c}= G_{d^{n-1} c} \circ \widehat{\Phi}_{d^{n-1} c},
\end{equation}
for all $n\geq 1$ and thus 
\begin{equation}\label{cover1}
\widehat{\Phi}_{d^n c} \circ \widehat{H}_{d^{n-1}c}\circ \cdots \circ \widehat{H}_c= G_{d^{n-1} c} \circ\cdots G_c\circ \widehat{\Phi}_{c}
\end{equation}
on $\widehat{\Omega}_{c}'$. 

\begin{prop}\label{Upto bihol}
For all $c>0$ sufficiently large, $\widehat{\Omega}_{c}'$ is biholomorphic to $\mathbb{C}\times \mathcal{A}_c$.
\end{prop}
\begin{proof}
First we prove that
\begin{align}\label{CAc}
\mathbb{C}\times \mathcal{A}_c=  \bigcup_{n=0}^\infty {\left(G_{d^{n-1}c}\circ \cdots\circ G_c\right)}^{-1}\left (\widehat{\Phi}_{d^n c}\left(\widehat{ \Omega_{d^n c}'\cap U_R^+}\right)\right)
\end{align}
(for $n=0$, we assume $G_{d^{n-1}c}\circ \cdots\circ G_c={\rm{Id}}$).
Let $G_{d^{n-1}c}\circ \cdots\circ G_c(s,t)\in \widehat{\Phi}_{d^n c}\left(\widehat{ \Omega_{d^n c}'\cap U_R^+ }\right)$. Then using Proposition\ \ref{UOmega} and (\ref{funct1}), we get that $t\in \mathcal{A}_c$. Thus
\[
 \bigcup_{n=0}^\infty {\left(G_{d^{n-1}c}\circ \cdots\circ G_c\right)}^{-1}\left (\widehat{\Phi}_{d^n c}\left(\widehat{ \Omega_{d^n c}'\cap U_R^+ }\right)\right) \subseteq \mathbb{C}\times \mathcal{A}_c.
\]
Now let $(s,t)\in \mathbb{C}\times \mathcal{A}_c$. Note that 
\[
G_{d^{n-1}c}\circ \cdots \circ G_c(s,t)=\left(({a^n}/{d^n})s+({a^{n-1}}/{d^{n-1}})Q(t)+\cdots+( {a}/{d})Q\left(t^{d^{n-2}}\right)+Q\left(t^{d^{n-1}}\right),t^{d^n}\right)
\]
for $n\geq 1$ and for $(s,t)\in \mathbb{C}\times \mathcal{A}_c$. Since $Q$ is a polynomial of degree $(d+1)$ by (\ref{STrho}), for large $n$, we have 
\[
G_{d^{n-1}c}\circ \cdots\circ G_c(s,t) \in \Phi(U_R^+)=\bigcup_{n=0}^\infty \Phi( \Omega_{d^nc}' \cap U_R^+ )=\bigcup_{n=0}^\infty \widehat{\Phi}_{d^n c}\left( \widehat{\Omega_{d^n c}'\cap U_R^+ }\right).
\]
Therefore, 
\[
G_{d^{n-1}c}\circ \cdots\circ G_c(s,t) \in \Phi( \Omega_{d^m r}' \cap U_R^+) ,
\]
for some $m$. Let $G_{d^{n-1}c}\circ \cdots\circ G_c(s,t)=\Phi(z)=(\psi(z), \phi^+(z))$.  Now if  $m>n$, then $G^+(z)>d^n c$ which implies $\lvert \phi^+(z)\rvert> \exp(d^n c)$. This is a contradiction, which in turn gives (\ref{CAc}).

\medskip 

Note that both (\ref{CAc}) and (\ref{UOmega c}) are increasing unions. Thus it is sufficient to prove that $\widehat{\Phi}_c$ is a bijection from $ {\left(\widehat{H}_c^n\right)}^{-1}\left(\widehat{ \Omega_{d^n c}'\cap U_R^+}\right)$ to ${\left(G_{d^{n-1}c}\circ \cdots\circ G_c\right)}^{-1}\left (\widehat{\Phi}_{d^n c}\left(\widehat{ \Omega_{d^n c}'\cap U_R^+ }\right)\right)$  for each $n\geq 0$. Using (\ref{cover}) inductively, we get 
\begin{align*}
\widehat{\Phi}_{d^n c}\left(\widehat{\Omega_{d^n c}' \cap U_R^+}\right)=G_{d^{n-1}c}\circ \cdots\circ G_c \circ \widehat{\Phi}_c\circ{ (\widehat{H}_c^n)}^{-1}\left(\widehat{\Omega_{d^n c}' \cap U_R^+}\right)\\
\Rightarrow {\left(G_{d^{n-1}c}\circ \cdots\circ G_c\right)}^{-1}\left(\widehat{\Phi}_{d^n c}\left(\widehat{\Omega_{d^n c}' \cap U_R^+}\right)\right)= \widehat{\Phi}_c\circ{ (\widehat{H}_c^n)}^{-1}\left(\widehat{\Omega_{d^n c}' \cap U_R^+}\right).
\end{align*}
This proves that $\widehat{\Phi}_c$ is surjective.
To prove that $\widehat{\Phi}_c$ is injective,   note that 
it follows from (\ref{cover1}) that
\[
\widehat{\Phi}_{d^n c} [H^n(z),H^n(C)]=\widehat{\Phi}_{d^n c} [H^n(z'),H^n(C')],
\]
for all $n \geq 0$ if  $\widehat{\Phi}_{c} [z,C]=\widehat{\Phi}_{c} [z',C']$, for some $[z,C], [z',C'] \in \widehat{\Omega}_{c}'$. Since $\Phi$ is injective on $U_R^+$, it must be the case that
$H^n(z)=H^n(z')$ for large $n$, and thus $z=z'$. Now since $\widehat{\Phi}_{c} [z,C]=\widehat{\Phi}_{c} [z',C']$, it follows from (\ref{defn phi hat}) that $[C{C'}^{-1}] \in \mathcal{H}_c$ which shows that $[z,C]=[z',C']$. This finishes the proof.
\end{proof} 



\section{Proof of Theorem \ref{AutOmegac}}

\no {\it{Step 1:}} Recall from Proposition\ \ref{Prop 1} that if  $a$ is an automorphism of $\Omega_c$, then $a$ also acts as an automorphism of $\Omega_c'$. Further, by Proposition\ \ref{Upto bihol}, the covering of $\Omega_c'$ corresponding to the subgroup $\mathbb{Z} \subseteq \mathbb{Z}[1/d]=\pi_1\left(\Omega_c'\right)$ is $ \mathbb{C}\times \mathcal{A}_c$ up to a biholomorphism.

\medskip 
\no 
{\it Claim:} Any automorphism of $\Omega_c'$  lifts as an automorphism of $ \mathbb{C}\times \mathcal{A}_c$.

\medskip
\no 
Let $a$ be an automorphism of $\Omega_c'$. Then $a$ lifts as an automorphism $A$ of $ \mathbb{C}\times \mathcal{A}_c$ if it induces $\pm \rm{Id}$ on 
$\pi_1\left(\Omega_c'\right)=\mathbb{Z}[{1}/{d}]$. It can be checked that any group isomorphism of $\mathbb{Z}[1/d]$ is of the form $i_a(x)=mx$, where $m=\pm d^s$ with $s\in \mathbb{Z}$. Thus there always exists an $s\in \mathbb{Z}$ such that $H^s \circ a$ (which is clearly a biholomorphism from $\Omega_c'$ to $\Omega_{d^s c}'$) induces identity as an isomorphism of $\mathbb{Z}[1/d]$. Therefore, $H^s \circ a$  lifts to a biholomorphism from $ \mathbb{C}\times \mathcal{A}_c$ to $ \mathbb{C}\times \mathcal{A}_{d^s c}$. Let 
\[
A(u,v)=\left(A_1(u,v), A_2(u,v)\right) : \mathbb{C}\times \mathcal{A}_c \rightarrow \mathbb{C}\times \mathcal{A}_{d^sc} 
\]
be this biholomorphic lift. For each $v\in \mathcal{A}_c$, the map 
\[
u \mapsto A_2(u,v)
\]
is a bounded entire function, which implies that $A_2(u,v)\equiv h(v)$ and thus
\[
A(u,v)=( A_1(u,v),h(v))
\]
for all $(u,v)\in  \mathbb{C}\times \mathcal{A}_c$. Similarly, 
\[
A^{-1}(u,v)=(A_1^*(u,v),h^*(v)),
\]
for all $(u,v)\in  \mathbb{C}\times \mathcal{A}_{d^sc}$.
Now since
\begin{align} \label{A2}
A\circ A^{-1}(u,v)= \left(A_1(A_1^*(u,v),h^*(v)),h h^*(v)\right)=(u,v) \nonumber\\
= \left(A_1^*( A_1(u,v),h(v)), h^* h(v)\right)=A^{-1}\circ A(u,v)
\end{align}
holds, it follows that 
\[
h\circ h^*\equiv h^*\circ h\equiv \rm{Id}
\]
in $\mathbb{C}$. Therefore, $h$ is a biholomorphism between $\mathcal{A}_c$ and $\mathcal{A}_{d^s c}$,  which is not possible unless $s=0$. The claim follows.

\medskip 
\no 
{\it{Step 2:}} Let $A(u,v)=\left(A_1(u,v), A_2(u,v)\right)$ be an automorphism of $\mathcal{A}_c \times \mathbb{C}$.
Exactly as in Step 1, one can show that
$$
A(u,v)=(A_1(u,v),h(v)),
$$
for all $(u,v)\in  \mathbb{C}\times \mathcal{A}_c$,  where $h(v)=\alpha v$ or $h(v)=\alpha {e^{c}}/{v}$ with $\lvert \alpha\rvert=1$.

\medskip

Now for a fixed $v\in \mathbb{C}$, consider the following pair of entire functions on $\mathbb{C}$: 
\[
P_1: u\mapsto A_1(u,h^*(v)) \text{ and } P_2: u\mapsto A_1^*(u,v).
\]
From (\ref{A2}), it follows that $P_1 \circ P_2(u) =u$, for all $u\in \mathbb{C}$.  Thus $P_1$ is onto. 
Now we prove that $P_1$ is also injective. Let $u_1\neq u_2$. Then since $A$ is an automorphism, for a fixed $v$, we have 
\[
A(u_1,h^*v)=(A_1(u_1,h^*v),v)\neq (A_1(u_2,h^*v),v)=A(u_2,h^*v),
\] 
which implies $A_1(u_1,h^*v)\neq A_1(u_2,h^*v)$, i.e., $P_1(u_1)\neq P_1(u_2)$. Thus $P_1$ is injective. Therefore for a fixed $v$,  $A_1(u,v)$ is an automorphism of $\mathbb{C}$ and thus
\begin{equation} \label{aut2}
A_1(u,v)=\beta(v)u+ \gamma (v), 
\end{equation}
where $\beta$ and $\gamma$ are holomorphic on $\mathbb{C}$.

\medskip 
\no 
{\it{Step 3:}}
It follows from (\ref{cover}) that 
\[
{\widehat{\Phi}_{d^n c}}^{-1} \circ G_{d^{n-1}c}= \widehat{H}_{d^{n-1} c}
\circ \widehat{\Phi}_{d^{n-1} c}^{-1},
\]
for each $n \geq 1$ and this gives
\[
\left(\widehat{\pi}_{d^n c} \circ \widehat{\Phi}_{d^n c}^{-1} \right)\circ G_{d^{n-1}c}=H\circ \left(\widehat{\pi}_{d^{n-1}c}\circ \widehat{\Phi}_{d^{n-1} c}^{-1}\right).
\]
Therefore, it follows from (\ref{lift}) that for each $n\geq 1$, $H:\Omega_{d^{n-1}c}' \rightarrow \Omega_{d^n c}'$ lifts to 
\[
\tilde{H}_{d^{n-1}c}: \mathbb{C}\times \mathcal{A}_{d^{n-1}c} \rightarrow  \mathbb{C}\times\mathcal{A}_{d^n c}
\]
which has the form 
\[
(z,\zeta)\mapsto \left(\frac{a}{d}z + Q(\zeta), \zeta^d\right).
\]

The following chain of commutative diagrams illustrates the various spaces and the maps between them.

\[\begin{tikzcd}
\mathbb{C}\times \mathcal{A}_c \arrow{r}{G_c}   \arrow[swap]{d}{\hat{\Phi}_c^{-1}} & \mathbb{C}\times \mathcal{A}_{dc} \arrow{r}{G_{dc}} \arrow{d}{\hat{\Phi}_{dc}^{-1}} &  \mathbb{C}\times \mathcal{A}_{d^2c} \arrow{r} {G_{d^2c}}  \arrow{d}{\hat{\Phi}_{d^2c}^{-1}}& \cdots\\
\hat{\Omega}_c' \arrow{r}{\hat{H}_c}\arrow[swap]{d}{\hat{\pi}_c} & \hat{\Omega}_{dc}'  \arrow{r}{\hat{H}_{dc}}  \arrow{d}{\hat{\pi}_{dc}} &  \hat{\Omega}_{d^2c}'  \arrow{r} {\tilde{H}_{d^2c}} \arrow{d}{\hat{\pi}_{d^2c}} & \cdots\\
\Omega_c' \arrow{r}{H} & \Omega_{dc}' \arrow{r} {H} & \Omega_{d^2c}'  \arrow{r}{H} & \cdots  \\
\end{tikzcd}
\]
Note that for each $m\geq 0$ and for each element $\left[{k}/{d^n}\right]\in \mathbb{Z}[1/d]/ \mathbb{Z}$, there exists a deck transformation $\gamma_{\frac{k}{d^n}}^{(m)}$ of $\mathbb{C}\times \mathcal{A}_{d^m c}$ for the covering map $\hat{\pi}_c \circ \hat{\Phi}_c^{-1}$.  Our goal is to identify these deck transformations. Further note that for each $m\geq 0$,  $\gamma_{1}^{(m)}\equiv \rm{Id}$ on  $\mathbb{C}\times \mathcal{A}_{d^m c}$.

\medskip 
\no 
{\it Claim:} For each $m, n\geq 0$,
\begin{equation}\label{comm cover}
\gamma_{\frac{1}{d^n}}^{(m+1)} \circ \tilde{H}_{d^m c}(z,\zeta) = \tilde{H}_{d^m c} \circ \gamma_{\frac{1}{d^{n+1}}}^{(m)}(z,\zeta),
\end{equation}
for all $(z,\zeta)\in \mathbb{C}\times \mathcal{A}_{d^m c}$. In other words, the following diagram commutes.

\[\begin{tikzcd}
\mathbb{C}\times \mathcal{A}_{d^m c} \arrow{r}{{\tilde{H}_{d^mc}}} \arrow[swap]{d}{\gamma^{(m)}_{{1}/{d^{n+1}}}} & \mathbb{C}\times \mathcal{A}_{d^{m+1} c} \arrow{d}{\gamma^{(m+1)}_{\frac{1}{d^{n}}}}  \\
\mathbb{C}\times \mathcal{A}_{d^m c} \arrow{r}{\tilde{H}_{d^mc}}& \mathbb{C}\times \mathcal{A}_{d^{m+1} c}
\end{tikzcd}
\]
 Let $(z,\zeta)\in \mathbb{C} \times \mathcal{A}_{d^m c}$ and let $\gamma_{{1}/{d^{n+1}}}^{(m)}$ send $(z,\zeta)$ to $(\tilde{z},\tilde{\zeta})$. Then $\tilde{H}_{d^m c}\circ \gamma_{{1}/{d^{n+1}}}^{(m)}$ sends $(z,\zeta)$ to  $\tilde{H}_{d^m c}(\tilde{z},\tilde{\zeta})$. Further, since $\gamma_{{1}/{d^{n+1}}}^{(m)}$ is completely determined by its action on any arbitrary element, the same holds for the map $\tilde{H}_{d^m c}\circ \gamma_{{1}/{d^{n+1}}}^{(m)}$. Also note that $\gamma_{{1}/{d^n}}^{(m+1)}$ sends $\tilde{H}_{d^m c}(z,\zeta)$ to $\tilde{H}_{d^m c}(\tilde{z},\tilde{\zeta})$. Thus (\ref{comm cover}) holds.  
 
 \medskip 
 
Now we prove that 
\begin{equation}\label{form aut 0}
\gamma_{\frac{1}{d^n}}^{(m)}
\begin{bmatrix}
 z \\ \zeta
\end{bmatrix}
=\begin{bmatrix} 
z+ \frac{d}{a} \sum_{l=0}^{n-1} {\left(\frac{d}{a}\right)}^l\left(Q(\zeta^{d^l})-Q\left( {\left(e^{\frac{2\pi i}{d^n}}\zeta)\right)}^{d^l} \right)\right)\\
 e^{\frac{2\pi i}{d^n}} \zeta
 \end{bmatrix},
\end{equation} 
for each $m,n\geq 0$ and the proof follows inductively. Recall that for each $m\geq 0$,  $\gamma_{1}^{(m)}\equiv \rm{Id}$ on  $\mathbb{C}\times \mathcal{A}_{d^m c}$.  It follows from (\ref{comm cover}) that 
\begin{equation} \label{induc1}
\gamma_1^{(m+1)} \circ \tilde{H}_{d^m c}(z,\zeta) = \tilde{H}_{d^m c} \circ \gamma_{{1}/{d}}^{(m)}(z,\zeta)
\end{equation}
for all $(z,\zeta)\in \mathbb{C}\times \mathcal{A}_{d^m c}$ and for all $m\geq 0$. Let $\gamma_{{1}/{d}}^{(m)}(z,\zeta)=\left(z_1^{(m)},\zeta_1^{(m)}\right)$, then from (\ref{induc1}), we have 
\[
\left(\frac{a}{d}z + Q(\zeta), \zeta^d\right)=\left(\frac{a}{d}z_1^{(m)} + Q\left(\zeta_1^{(m)}\right), {\left(\zeta_1^{(m)}\right)}^d\right).
\]
Comparing both sides of the above equation, we have 
\[
\zeta_1^{(m)}=e^{\frac{2\pi i}{d}}\zeta
\]
and 
\[
z_1^{(m)}=z+\frac{d}{a}Q(\zeta)-\frac{d}{a}Q\left(e^{\frac{2\pi i}{d}}\zeta\right).
\]
Thus (\ref{form aut 0}) holds for $n=1$. Now let (\ref{form aut 0}) holds for some $n\geq 1$, we prove that the same holds for $n+1$. Let 
\begin{equation*}
\gamma_{\frac{1}{d^n}}^{(m)}
\begin{bmatrix}
 z \\ \zeta
\end{bmatrix}
=\begin{bmatrix} 
z+ \frac{d}{a} \sum_{l=0}^{n-1} {\left(\frac{d}{a}\right)}^l\left(Q(\zeta^{d^l})-Q\left( {\left(e^{\frac{2\pi i}{d^n}}\zeta)\right)}^{d^l} \right)\right)\\
 e^{\frac{2\pi i}{d^n}} \zeta
 \end{bmatrix}
\end{equation*}
for all $m\geq 0$. Then (\ref{comm cover}) gives
\begin{equation*}
\gamma_{\frac{1}{d^n}}^{(m+1)} \circ \tilde{H}_{d^m c}(z,\zeta) = \tilde{H}_{d^m c} \circ \gamma_{\frac{1}{d^{n+1}}}^{(m)}(z,\zeta),
\end{equation*}
for all $(z,\zeta)\in \mathbb{C}\times \mathcal{A}_{d^m c}$. Once we use the explicit expressions of 
$ \tilde{H}_{d^m c}$ and $\gamma_{\frac{1}{d^n}}^{(m+1)}$, as before we can extract the expression of $\gamma_{\frac{1}{d^{n+1}}}^{(m)}$ which turns out to be 
\begin{equation*}
\gamma_{\frac{1}{d^{n+1}}}^{(m)}
\begin{bmatrix}
 z \\ \zeta
\end{bmatrix}
=\begin{bmatrix} 
z+ \frac{d}{a} \sum_{l=0}^{n} {\left(\frac{d}{a}\right)}^l\left(Q(\zeta^{d^l})-Q\left( {\left(e^{\frac{2\pi i}{d^{n+1}}}\zeta)\right)}^{d^l} \right)\right)\\
 e^{\frac{2\pi i}{d^{n+1}}} \zeta
 \end{bmatrix}.
\end{equation*} 
Thus we prove (\ref{form aut 0}). 

Next we prove 
\begin{equation}\label{form aut}
\gamma_{\frac{k}{d^n}}^{(m)}
\begin{bmatrix}
 z \\ \zeta
\end{bmatrix}
=\begin{bmatrix} 
 z+ \frac{d}{a} \sum_{l=0}^{n-1} {\left(\frac{d}{a}\right)}^l\left(Q(\zeta^{d^l})-Q\left( {\left(e^{\frac{2 k\pi i}{d^n}}\zeta)\right)}^{d^l} \right)\right)\\
 e^{\frac{2 k\pi i}{d^n}}\zeta 
 \end{bmatrix},
\end{equation}
for each $m,n\geq 0$ and $k\geq 1$.  Note that we just proved that (\ref{form aut}) is true for $k=1$. Let us assume that (\ref{form aut}) holds for some $k\geq 1$. Then 
\begin{align*}
\gamma_{\frac{k+1}{d^n}}^{(m)}\begin{bmatrix}
 z \\ \zeta
\end{bmatrix}=\left(\gamma_{\frac{1}{d^n}}^{(m)}\circ\gamma_{\frac{k}{d^n}}^{(m)}\right)
\begin{bmatrix}
 z \\ \zeta
\end{bmatrix}
=\gamma_{\frac{1}{d^n}}^{(m)}\begin{bmatrix} 
 z+ \frac{d}{a} \sum_{l=0}^{n-1} {\left(\frac{d}{a}\right)}^l\left(Q(\zeta^{d^l})-Q\left( {\left(e^{\frac{2 k\pi i}{d^n}}\zeta)\right)}^{d^l} \right)\right)\\
 e^{\frac{2 k\pi i}{d^n}}\zeta 
 \end{bmatrix}\\
=\begin{bmatrix} 
 z+ \frac{d}{a} \sum_{l=0}^{n-1} {\left(\frac{d}{a}\right)}^l\left(Q(\zeta^{d^l})-Q\left( {\left(e^{\frac{2 (k+1)\pi i}{d^n}}\zeta)\right)}^{d^l} \right)\right)\\
 e^{\frac{2 (k+1)\pi i}{d^n}}\zeta 
 \end{bmatrix} .
\end{align*}
Thus we prove that (\ref{form aut}) is true for all $k\geq 1$.

Therefore, if $\tilde{p}=(z,\zeta)\in \mathbb{C} \times  \mathcal{A}_{d^m c}$ is in the fiber of any point $p\in \Omega_{d^m c}'$, then the other points in the fiber  of $p$ are precisely of the form $\gamma_{{k}/{d^n}}^{(m)}(z,\zeta)$ for some $n\geq 0$ and $k\geq 1$. 

\medskip 
\no 
{\it Step 4:}  Let $(u,v)$ and $(u',v')$ be in the same fiber. Then clearly, 
$${(v'/v)}^{d^n}=1,$$
for some $n \ge 1$. 

\medskip 
\no 
{\it Case 1:} Suppose that 
\[
A(z,\zeta)=( \beta (\zeta) z+ \gamma(\zeta), \alpha \zeta)
\]
for $(z,\zeta)\in \mathbb{C}\times \mathcal{A}_c$. Using (\ref{form aut}), it follows that 
\[
u'=u+\frac{d}{a} \sum_{l=0}^{n-1} {\left(\frac{d}{a}\right)}^l\left(Q\left(v^{d^l}\right)-Q\left( {\left(e^{\frac{2 k\pi i}{d^n}}v\right)}^{d^l} \right)\right)
\]
and 
\[
v'=e^{\frac{2 k\pi i}{d^n} }v.
\]
Thus
\[
A_2(u',v')=\alpha v'=e^{\frac{2 k\pi i}{d^n} }(\alpha  v)=e^{\frac{2 k\pi i}{d^n} } A_2(u,v)
\]
and 
\begin{align}\label{A2uv}
A_1(u',v')-A_1(u,v)= 
\left(\beta(v')-\beta(v)\right)u + 
\beta(v')\frac{d}{a} \sum_{l=0}^{n-1} {\left(\frac{d}{a}\right)}^l\left(Q\left(v^{d^l}\right)-Q\left( {\left(e^{\frac{2 k\pi i}{d^n}}v\right)}^{d^l} \right)\right) \nonumber \\
+\gamma(v')-\gamma(v).
\end{align}
Now since $A(u,v)$ and $A(u',v')$ are in the same fiber, we have 
\[
A_1(u,v)-A_1(u',v')=\Delta (\alpha v, \alpha v').
\]
Consequently, it follows from (\ref{A2uv}) that $\beta(v)=\beta(v')$, or in other words, 
\[
\beta(v)=\beta \left(v e^\frac{2\pi i k}{d^n}\right),
\]
for all $k\geq 1$ and for all $n\geq 0$. Therefore, $\beta(v)\equiv \beta$ in $\mathbb{C}$. Thus, from (\ref{A2uv}), it follows that
\begin{equation}\label{alpha u}
\Delta (\alpha v, \alpha v')=\beta\frac{d}{a} \sum_{l=0}^{n-1} {\left(\frac{d}{a}\right)}^l\left(Q\left(v^{d^l}\right)-Q\left( {\left(e^{\frac{2 k\pi i}{d^n}}v\right)}^{d^l} \right)\right)+\gamma(v')-\gamma(v).
\end{equation}
Now 
\begin{equation}\label{alpha uv} 
A_1(u',v')-A_1(u,v)= \frac{d}{a} \sum_{l=0}^{n-1} {\left(\frac{d}{a}\right)}^l\left(Q\left(\alpha^{d^l} v^{d^l}\right)-Q\left( {\left(e^{\frac{2 k\pi i}{d^n}}\alpha v\right)}^{d^l} \right)\right).
\end{equation}
Combining (\ref{alpha u}) and (\ref{alpha uv}), we get that the modulus of 
\begin{align*}
&\beta\frac{d}{a} \sum_{l=0}^{n-1} {\left(\frac{d}{a}\right)}^l\left(Q\left(v^{d^l}\right)-Q\left( {\left(e^{\frac{2 k\pi i}{d^n}}v\right)}^{d^l} \right)\right)-\frac{d}{a} \sum_{l=0}^{n-1} {\left(\frac{d}{a}\right)}^l\left(Q\left(\alpha^{d^l} v^{d^l}\right)-Q\left( {\left(e^{\frac{2 k\pi i}{d^n}}\alpha v\right)}^{d^l} \right)\right)\\
&=
\frac{d}{a} \sum_{l=0}^{n-1} {\left(\frac{d}{a}\right)}^l\left[\beta Q\left(v^{d^l}\right)-Q\left(\alpha^{d^l} v^{d^l}\right)\right]-\frac{d}{a} \sum_{l=0}^{n-1} {\left(\frac{d}{a}\right)}^l\left[\beta Q\left( {\left(e^{\frac{2 k\pi i}{d^n}}v\right)}^{d^l} \right)- Q\left({\left(e^{\frac{2 k\pi i}{d^n}}\alpha v\right)}^{d^l} \right)\right]\\
&=
{\left(\frac{d}{a}\right)}^{n-1} \left[ \beta {v}^{d^{n-1}(d+1)}-{\left(\alpha v \right)}^{d^{n-1} (d+1)}-
  \beta {\left(e^{\frac{2 k\pi i}{d^n}}v \right)}^{d^{n-1}(d+1)}+{\left(\alpha  v  e^{\frac{2 k\pi i}{d^n}}\right)}^{d^{n-1}(d+1)}\right]
+ l.o.t.,
\end{align*}
is uniformly bounded by some fixed constant, for all $n\geq 0$ and  for any fixed $v$ with $1<\lvert v \rvert< e^c$. Therefore, 
\begin{equation*} \label{est beta}
\gamma(v)-\gamma(v')= {\left(\frac{d}{a}\right)}^{n-1} \left(\beta-\alpha^{d^{n-1}(d+1)}\right) v^{d^{n-1}(d+1)}\left( 1-e^{\frac{2k\pi i}{d}}\right)\left[1+ L_{n-1}(v)\right] 
\end{equation*}
is bounded, where $L_{n-1}$ is a Laurent polynomial of degree at most $d^n$. Now, we claim that
\begin{equation}\label{beta one}
a_n=\alpha^{d^{n}(d+1)} \rightarrow  \beta,
\end{equation}
as $n \rightarrow \infty$.  If not, there exists a sequence $n_k \rightarrow \infty$ such that 
\begin{equation*}\label{alphaa beta}
\left\lvert \alpha^{d^{n_k}(d+1)}- \beta\right\rvert >\delta>0,
\end{equation*}
for all $k\geq 1$. Thus, 
\begin{equation*} \label{alpha beta}
\left\lvert L_{n_k}(v)\right\rvert \leq \frac{(n_k+1)K_1}{{\lvert v \rvert}^{d^{n_k}}},
\end{equation*} 
 for some $K_1>1$, if $\lvert {d}/{a} \rvert\geq1$ and  
\begin{equation*} \label{alpha beta}
\lvert L_{n_k}(v)\rvert \leq \frac{(n_k+1)K_2}{{\lvert v \rvert}^{d^{n_k}}}{\left(\frac{a}{d}\right)}^{n_k},
\end{equation*} 
for some $K_2>1$, if $\lvert {d}/{a} \rvert <1$. This leads to a contradiction since $\gamma(v)-\gamma(v')$ does not remain bounded in either case. Thus, (\ref{beta one}) holds. Also, 
${a_{n+1}}/{a_n} \rightarrow 1$ as $n\rightarrow \infty$. This shows that 
\begin{equation*}
\alpha^{(d+1)(d-1)d^{n}} \rightarrow  1,
\end{equation*}
as  $n\rightarrow \infty$. Thus, from (\ref{beta one}), it follows  that $\beta^{d-1}=1$. Now since ${\left(\alpha^{d+1}\right)}^{d^n} \rightarrow \beta$ and $\beta$ is a repelling fixed point for the map $z\mapsto z^d$, we get 
\[
\alpha^{d+1}=\beta.
\]
Therefore (\ref{est beta}) shows that $\gamma(v)=\gamma(v')$, which in turn gives $\gamma\equiv k$, for some $k\in \mathbb{C}$.

\medskip 
\no 
{\it Case 2:} Now let 
\begin{equation}\label{aut2}
A(z,\zeta)=( \beta (\zeta) z+ \gamma(\zeta), \alpha {e^c}/{\zeta}),
\end{equation}
 for $(z,\zeta)\in \mathbb{C}\times \mathcal{A}_c$. As in Case 1, we can show that $\beta(u)\equiv \beta$, for some constant $\beta\in \mathbb{C}$. Further,
we can show that if $(u,v)$ and $(u',v')$ are in the same fiber and $e^{c}<{\lvert v\rvert}^2<e^{2c}$, then the modulus of 
\begin{align*}
&\beta\frac{d}{a} \sum_{l=0}^{n-1} {\left(\frac{d}{a}\right)}^l\left(Q\left(v^{d^l}\right)-Q\left( {\left(e^{\frac{2 k\pi i}{d^n}}v\right)}^{d^l} \right)\right)-\frac{d}{a} \sum_{l=0}^{n-1} {\left(\frac{d}{a}\right)}^l\left(Q\left(\left({\frac{\alpha e^c}{v}}\right)^{d^l}\right)-Q\left( {\left(e^{\frac{2 k\pi i}{d^n}}\frac{\alpha e^c}{v}\right)}^{d^l} \right)\right)\\
&=
\beta {\left(\frac{d}{a}\right)}^{n-1} v^{d^{n-1}(d+1)}\left(1-e^{\frac{2k\pi i}{d}}\right)\left[1+L_{v}(n-1)+T_{c}(n-1)\right],
\end{align*}
 is uniformly bounded by some fixed constant for all $n\geq 0$. Here, $L_{v}(n-1)$ and $T_{c}(n-1)$ come from the first and second parts of the equation, and
\begin{equation*} 
 L_{v}(n-1) =O\left( \frac{n}{{\lvert v \rvert}^{d^{n-1}}} \right) \text{ and } T_{c}(n-1)={\left(\frac{\alpha e^c}{v^2}\right)}^{d^{n-1}(d+1)}+O \left(\frac{n}{{\lvert v\rvert}^{d^{n-1}}}\right),
\end{equation*}
if $\lvert {d}/{a} \rvert >1$ and
\begin{equation*} 
L_{v}(n-1)= O\left(\frac{n}{{\lvert v \rvert}^{d^{n-1}}}{\left(\frac{a}{d}\right)}^{n-1}\right) \text{ and } T_{c}(n-1)={\left(\frac{\alpha e^c}{v^2}\right)}^{d^{n-1}(d+1)}+O \left(\frac{n}{{\lvert v\rvert}^{d^{n-1}}} {\left(\frac{a}{d}\right)}^{n-1}\right),
\end{equation*}  if $\lvert {d}/{a} \rvert <1$. This implies $\beta=0$ and this is clearly not possible. Thus $A$ cannot be of the form (\ref{aut2}). 

That $A$ cannot be of the form (\ref{aut2}) can be shown in a rather direct way using the definition of the projection map $\tilde{\pi}_c=\widehat{\pi}_c\circ \hat{\Phi}_c^{-1}$ from $\mathbb{C}\times \mathcal{A}_c$ to $\Omega_c'$. The proof we present below  is inspired by a comment made by one of the referees. Let $(z,\zeta) \in \mathbb{C} \times \mathcal{A}_c$ and let $\tilde{\pi}_c(z,\zeta)\in \{G_H^+=r\}$ for some $0<r<c$. It follows from the definitions of $\widehat{\pi}_c$ and $\hat{\Phi}_c$ that if $\hat{\Phi}_c^{-1}(z,\zeta)=[z_\zeta,C_{z,\zeta}]$, then $z_\zeta\in \{G_H^+=r\}$. Now by definition
\[
(z,\zeta)=\hat{\Phi}_c\left([z_\zeta,C_{z,\zeta}]\right)= \left(\widehat{\psi}_c([z_\zeta,C_{z,\zeta}]), \widehat{\phi}_c^+([z
_\zeta,C_{z,\zeta}])\right)
\]
and since there exists some $n\in \mathbb{N}$ such that $H^n(z_\zeta)\in V_R^+$, we have 
\[
\widehat{\phi}^+_{d^n c}[H^n (z_\zeta), H^n \left(C_{z,\zeta}\right)]={\left( \widehat{\phi}^+_c [z_\zeta
,C_{z,\zeta}]\right)}^{d^n}=\phi^+ \circ H^n (z_\zeta).
\]
Thus since $G_H^+\equiv \log \lvert\phi^+ \rvert$ in $V_R^+$, we have 
$$
\left\lvert \widehat{\phi}^+_c [z_\zeta,C_{z,\zeta}]\right\rvert=e ^r.
$$
Note that since $\tilde{\pi}_c \circ A=a\circ \tilde{\pi}_c $ and we have $a\left(\{G_H^+=r\}\right)=\{G_H^+=r\}$  by Proposition \ref{Prop 1}, it follows that $a\circ \tilde{\pi}_c (z,\zeta)=\tilde{\pi}_c \circ A(z,\zeta)=\widehat{\pi}_c\circ \hat{\Phi}_c^{-1}\circ A (z,\zeta) \in \{G_H^+=r\}$. Let $\hat{\Phi}_c^{-1}\circ A (z,\zeta)=[z_\zeta',C_{z,\zeta}']$. Then as above we can show 
\[
A(z,\zeta)=\hat{\Phi}_c\left([z_\zeta',C_{z,\zeta}']\right)= \left(\widehat{\psi}_c([z_\zeta',C_{z,\zeta}']), \widehat{\phi}_c^+([z
_\zeta',C_{z,\zeta}'])\right)
\]
with $$
\left\lvert \widehat{\phi}^+_c [z_\zeta',C_{z,\zeta}']\right\rvert=e ^r.
$$
Thus since $A$ is of the form (\ref{aut2}), we have  $e^r={e^c}/{e^r}$. Thus $r={c}/{2}$. So if we start with any $r\neq {c}/{2}$, we get a contradiction.

\medskip 
 
Therefore, each $a\in \rm{Aut}(\Omega_c')$ lifts to some $A\in \rm{Aut}(\mathbb{C}\times \mathcal{A}_c)$ of the form
\begin{equation}\label{aut lift}
(z,\zeta)\mapsto (\beta z+\gamma, \alpha \zeta),
\end{equation}
for $(z,\zeta)\in \mathbb{C}\times \mathcal{A}_c$, where $\beta^{d-1}=1$ and $\alpha^{d+1}=\beta$.

\medskip 
\no 
{\it Step 5:} 
Let $a$ be an automorphism of $\Omega_c'$. Then $a$ lifts as an automorphism $A$ of $ \mathbb{C}\times \mathcal{A}_c$, which is of the form (\ref{aut lift}). We claim that $a$ has a unique lift.  If not, then let
\begin{equation*}
A'(z,\zeta)=(\beta' z+\gamma', \alpha' \zeta),
\end{equation*}
for all $(z,\zeta)\in \mathbb{C}\times \mathcal{A}_c$, with $\beta'^{d-1}=1$ and $\alpha'^{d+1}=\beta'$, be another lift of $a$. Thus,  $(\beta z+\gamma, \alpha \zeta)$ and $(\beta' z+\gamma', \alpha' \zeta)$ are in the same fiber, for all $(z,\zeta)\in \mathbb{C}\times \mathcal{A}_c$.  Therefore, 
$
\alpha' \zeta= e^{\frac{2\pi k i}{d^n}} \alpha \zeta$, for all $\zeta\in \mathcal{A}_c$, which implies  ${\left({\alpha'}/{\alpha}\right)}^{d^n}=1.
$
Also we have ${\left({\alpha'}/{\alpha}\right)}^{d^2-1}=1$. Thus, $\alpha'=\alpha$. On the other hand, (\ref{form aut}) shows that
 \begin{equation}\label{beta gamma}
 (\beta'-\beta) z+(\gamma'-\gamma)= \frac{d}{a}\sum_{l=0}^{n-1} {\left(\frac{d}{a}\right)}^l\left(Q\left(\left({\alpha \zeta}\right)^{d^l}\right)-Q\left( {\left(e^{\frac{2 k\pi i}{d^n}}\alpha\zeta\right)}^{d^l} \right)\right),
 \end{equation}
 for all $(z,\zeta)\in  \mathbb{C}\times \mathcal{A}_c$.  Since the right side of (\ref{beta gamma}) is a polynomial in  $\zeta$ without constant term, we must have $\beta'=\beta$ and 
$\gamma'=\gamma$.  This shows that $a$ has a unique lift. 
 
\medskip 
 
Now let  $A(z,\zeta)=(z+\gamma, \zeta)$ be any automorphism of $ \mathbb{C}\times \mathcal{A}_c$. Note that $\Omega_c'=\tilde{\pi}_c(\mathbb{C}\times \mathcal{A}_c)$, where $\tilde{\pi}_c=\widehat{\pi}_c\circ \hat{\Phi}_c^{-1}$. We define $a: \Omega_c' \rightarrow \Omega_c'$ as follows:
$$
a\left(\tilde{\pi}_c(z,\zeta)\right)=\tilde{\pi}_c( z+\gamma, \zeta),
$$
for all $(z,\zeta)\in \mathbb{C} \times \mathcal{A}_c$. First we check that $a$ is well-defined. Let $\tilde{\pi}_c(z,\zeta)=\tilde{\pi}_c(z',\zeta')$. Then
\begin{equation*}
\begin{bmatrix}
 z' \\ \zeta'
\end{bmatrix}
=\begin{bmatrix} 
 z+ \frac{d}{a} \sum_{l=0}^{n-1} {\left(\frac{d}{a}\right)}^l\left(Q(\zeta^{d^l})-Q\left( {\left(e^{\frac{2k\pi i}{d^n}}\zeta\right)}^{d^l} \right)\right)\\
 e^{\frac{2k\pi i}{d^n}}\zeta
 \end{bmatrix},
\end{equation*} 
which implies 
\begin{equation*}
\begin{bmatrix}
   z'+\gamma\\
   \zeta'
\end{bmatrix}
=\begin{bmatrix} 
  z+\gamma+ \frac{d}{a} \sum_{l=0}^{n-1} {\left(\frac{d}{a}\right)}^l\left(Q\left({(\zeta)}^{d^l}\right)-Q\left( {\left(e^{\frac{2k\pi i}{d^n}}\zeta\right)}^{d^l} \right)\right)\\
 e^{\frac{2k\pi i}{d^n}} \zeta 
 \end{bmatrix}.
\end{equation*}
By (\ref{form aut}), $( z+\gamma,  \zeta)$ and $( z+\gamma', \zeta)$ are in the same fiber, which proves that  $a$ is well-defined. Also, $a$ is clearly a bijection. Therefore, 
$\mathbb{C} \subseteq \rm{Aut}(\Omega_c')$ $\subseteq \mathbb{Z}_{d^2-1}\times \mathbb{C}$.

\section{Proof of Theorem \ref{bihol short} and Corollary \ref{cor short}}

Let $\phi : \Omega_{c_1} \rightarrow \Omega_{c_2}$ be a biholomorphism. By Corollary (\ref{Omega cd}), $\phi(K^+)=K^+$. Then, $\phi$ acts as a biholomorphism between the corresponding punctured Short $\mathbb{C}^2$'s $\Omega_{c_1}'$ and $\Omega_{c_2}'$. By Step 1, Section 6, we obtain $c_1=c_2 d^{\pm n}$, for some integer $n\ge 1$. Conversely, if $c_1=c_2 d^{\pm n}$, then $H^{\pm n}$ is a biholomorphism between $\Omega_{c_1}$ and $\Omega_{c_2}$. This completes the proof of 
Theorem \ref{bihol short}. 

\medskip 

For any $c>0$, the family $\left\{\Omega_{\ell} \right\}_{{c}/{d}< \ell < c}$ is a continuum of pairwise non-biholomorphic Short $\mathbb{C}^2$'s. Now by Proposition \ref{Prop 1}, 
$$
\rm{Aut}\left(\Omega_c\right)\subseteq  \rm{Aut}\left(\Omega_{\ell}\right)\subseteq  \rm{Aut}\left(\Omega_{c/d} \right)=\rm{Aut}\left(\Omega_c\right),
$$
for all ${c}/{d}<\ell < c$. This shows that each member of this family has the same automorphism group.

The following remark is motivated by a question asked by one of the referees.
\begin{rem}
Since $H$ is an automorphism of $\mathbb{C}^2$, $H^{-1}$ exists and it is easy to see that  $H^{-1}(x,y) = ({(p(x)-y)}/{a}, x)$.  As $G_H^+$, one can construct
    \[
    G^-_H:= \lim_{n \ra \infty} \frac{1}{d^n} \log^+ \Vert H^{-n} \Vert
    \]
   which is  continuous, non-negative plurisubharmonic function on $\mbb C^2$.  For any $c > 0$ the sub-level set $\{G_H^- < c\}$ is also a {\it Short}\;$\mbb C^2$.

Further, note that similar arguments as in the proof of Proposition \ref{Prop 1}, gives that if $\Omega_c$ is biholomorphic to $\Omega_{c'}^{-}=\{z \in \mathbb{C}^2: G_H^-(z)<c'\}$ for some $c,c'>0$, then the biholomorphism preserves the non-escaping sets, i.e., if $\phi: \Omega_{c} \to \Omega_{c'}^{-}$ is a biholomorphism then $\phi(K^+)=K^-.$ Thus, if $H$ is a hyperbolic H\'{e}non map that is not volume preserving, this situation will never arise (see \cite[Lemma 5.5] {BS1} and \cite[Theorem 5.6]{BS1}). Also using a similar set of arguments as in Step 1 of Section 6, one can prove that if $\Omega_{c'}^{-}$ is biholomorphic to $\Omega_{c}^{-}$, then $c'=d^{\pm n}c$,
for some $n\in \mathbb{Z}$.
\end{rem}

Below we give an explicit example of a volume preserving H\'{e}non map such that $\Om_c$ and $\Om_{c}^-$ is biholomorphic, for every $c>0.$
\begin{exam}
Let $H(x,y)=(y, y^2-x)$ and $\ti I(x,y)=(y, x)=\ti I^{-1}(x,y)$. Then
 \[H^{-1}(x,y)=\ti I^{-1} \circ H \circ \ti I(x,y)=(x^2-y,x).\]
 Since $\|H^n(x,y)\|=\|H^{-n} \circ \ti I(x,y)\|$ for every $n \ge 1$, $G_H^+(x,y) =G_H^- \circ \ti I(x,y).$ Thus, $\ti I (\Om_c) =\Om^-_c$ for every $c\ge 0$ and $\Om_c \cong \Om_c^-.$
\end{exam}


\bibliographystyle{amsplain}

\end{document}